\documentclass[10pt,reqno]{amsart}

\usepackage{amsthm,amsmath,amstext,amssymb,amscd,euscript, mathrsfs, dsfont,multicol,times,enumerate,subfig,sidecap}
\usepackage{enumerate}
\usepackage{amsmath, amssymb, amsthm}
\usepackage{mathrsfs}
\usepackage{esint}
\usepackage{xcolor}
\usepackage{mathtools}
\usepackage{dirtytalk} %used for quotation marks using \say{}
\usepackage{verbatim}
\usepackage{todonotes}
\usepackage{xcolor}
\usepackage{bm}
\usepackage{bbm}

\usepackage[colorlinks=true, pdfstartview=FitV, linkcolor=blue, citecolor=red, urlcolor=blue]{hyperref} %%% % Includes hyperrefs into pdf

\newtheorem{thm}{Theorem}[section]

\newtheorem{lem}[thm]{Lemma}

\theoremstyle{definition}
\newtheorem{defn}[thm]{Definition}

\newtheorem{assumption}[thm]{Assumption}
\newtheorem{rem}[thm]{Remark}

\newcommand\bE{\mathbb{E}}

\newcommand\bN{\mathbb{N}}

\newcommand\bP{\mathbb{P}}

\newcommand\bR{\mathbb{R}}

\newcommand\bT{\mathbb{T}}

\newcommand\cpar{\text{$[$\kern-.38em$[$}}
\newcommand\cbrk{\text{$]$\kern-.15em$]$}}
\newcommand\opar{\text{\,\raise.2ex\hbox{${\scriptstyle
				|}$}\kern-.34em$($}}
\newcommand\obrk{\text{$)$\kern-.34em\raise.2ex\hbox{${\scriptstyle |}$}}\,}
\renewcommand{\bar}{\overline}

\renewcommand{\d}{\mathrm{d}}

\renewcommand{\epsilon}{\varepsilon}

\def\XXint#1#2#3{{\setbox0=\hbox{$#1{#2#3}{\int}$ }
		\vcenter{\hbox{$#2#3$ }}\kern-.58\wd0}}

\makeatletter
\def\@tocline#1#2#3#4#5#6#7{\relax
	\ifnum #1>\c@tocdepth % then omit
	\else
	\par \addpenalty\@secpenalty\addvspace{#2}%
	\begingroup \hyphenpenalty\@M
	\@ifempty{#4}{%
		\@tempdima\csname r@tocindent\number#1\endcsname\relax
	}{%
		\@tempdima#4\relax
	}%
	\parindent\z@ \leftskip#3\relax \advance\leftskip\@tempdima\relax
	\rightskip\@pnumwidth plus4em \parfillskip-\@pnumwidth
	#5\leavevmode\hskip-\@tempdima
	\ifcase #1
	\or\or \hskip 1em \or \hskip 2em \else \hskip 3em \fi%
	#6\nobreak\relax
	\dotfill\hbox to\@pnumwidth{\@tocpagenum{#7}}\par
	\nobreak
	\endgroup
	\fi}
\makeatother

\begin{document}
	\title{The stochastic Keller--Segel system in critical spaces}
	
	\author[Jae-Hwan Choi]{Jae-Hwan Choi}
	\address[J.-H. Choi]{School of Mathematics, Korea Institute for Advanced Study, 85 Hoegi-ro, Dongdaemun-gu, Seoul, 02455, Republic of Korea}
	\email{jhchoi@kias.re.kr}
	
	\author[Ankit Kumar]{Ankit Kumar}
	\address[A. Kumar]{Montanuniversit\"{a}t Leoben,
		Department Mathematik und Informationstechnologie,
		Franz Josef Stra{\ss}e 18,
		8700 Leoben,
		Austria}
	\email{ankit.kumar@unileoben.ac.at}
	
	\author[Andrea Pitrone]{Andrea Pitrone}
	\address[A. Pitrone]{Mathematical Institute, University of Oxford, OX26GG Oxford, United Kingdom}
	\email{andrea.pitrone@maths.ox.ac.uk}
	
	\author[Shyam Popat]{Shyam Popat}
	\address[S. Popat]{Centre de Math\'ematiques Appliqu\'ees (CMAP), \'Ecole Polytechnique, 91120 Palaiseau, France}
	\email{shyam.popat@polytechnique.edu}
	
	\author[Max Sauerbrey]{Max Sauerbrey}
	\address[M. Sauerbrey]{Max--Planck--Institut f\"ur Mathematik in den Naturwissenschaften,
		04103 Leipzig, Germany}
	\email{maxsauerbrey97@gmail.com}

	\thanks{}
	
	

	\maketitle
	
	\begin{abstract}
		
		We study stochastic, parabolic-parabolic Keller--Segel equations on the $d$-dimensional torus {in} scaling critical Besov spaces, for $d \geq 3$. Using stochastic maximal regularity estimates, we prove local well-posedness of the equation, i.e., that there exists a unique solution up to a maximal time of existence.
		We show that the time of existence can be made arbitrarily large with arbitrarily high probability, provided that the initial data is sufficiently small in these critical spaces.
	\end{abstract}
	\vspace{10pt}

	\section{Introduction}
	\subsection{Outline}
	In this paper, we study the well-posedness of the stochastic Keller--Segel system
	\begin{equation}\label{25.06.19.16.01}
		\left\{
		\begin{aligned}
			\mathrm{d}u &= \big(\Delta u\,-\,\mathrm{div}(u\nabla v)\big)\mathrm{d}t\,+\,\sum_{n\geq 1}g_n^{1}(u,v ,\nabla v)\,\mathrm{d}W^1_n,\\
			\mathrm{d}v&=(\Delta v\,+\,u\,-\,v)\,\mathrm{d}t\,+\,\sum_{n\geq 1}g^2_n(v)\,\mathrm{d}W_n^2,\\
		\end{aligned}
		\right.
	\end{equation}
	on the $d$-dimensional torus for dimensions $d\geq 3$, supplemented with a suitable initial condition 
	\begin{equation}
		\label{25.06.19.16.02}
		(u(0),v(0)) \,=\, (u_0,v_0).
	\end{equation}
	In the above equation,  $(W^1_n)_{n\geq1}$, $(W^2_n)_{n\geq1}$ denote families of standard independent Brownian motions, and $(g^1_n)_{n\geq 1}, (g^2_n)_{n\geq 1}$ denote families of sufficiently regular nonlinearities. {Using the methods from \cite{AV_SQEE_pt1,AV_SQEE_pt2}} based on stochastic maximal $L^p_\kappa L^q$-regularity estimates, we show 
	in scaling-critical spaces that:
	
	\begin{enumerate}[(I)]
		\item The stochastic Keller--Segel  system \eqref{25.06.19.16.01} admits a maximal unique solution (Theorem \ref{thm:local_wp}).
		\item Small data implies that the time of existence is large, with high probability (Theorem \ref{thm_global_sol}).
	\end{enumerate}
	
	\subsection{Deterministic Keller--Segel equations}
	Chemotaxis is the directed movement of organisms triggered by a chemical concentration gradient and exhibited by many types of motile cells. Amoeboid cells, for instance, migrate toward food sources, and immune cells such as macrophages and neutrophils move toward invading pathogens. Other cells involved in immune responses and wound healing are drawn to inflamed tissues by chemical cues.
	
	The theoretical and mathematical study of chemotaxis goes back to Patlak’s work \cite{CP53} in the 1950s and to Keller and Segel \cite{KS70} in the 1970s. A basic, parabolic-parabolic version of the model is 
	\begin{equation}\label{eqn1}
		\left\{\begin{aligned}
			\d u &=\big[r_u \Delta u-\chi \mathrm{div} \big(u\nabla v\big)\big]\d t, 
			\\
			\d v&=\big[ r_v \Delta v+\beta u-\alpha v\big]\d t,
		\end{aligned}
		\right.
	\end{equation}
	in which  $u$ represents the cell density and $v$  the concentration of the chemical signal. The positive parameters $r_u$ and $r_v$ are the diffusivities of the cells and the chemoattractant, respectively. The positive constant $\chi$ measures chemotactic sensitivity, and in the equation for the signal, $\alpha \ge 0$ is the damping rate and $\beta\ge 0$ is the rate at which the cells produce the chemoattractant. We remark that while \eqref{eqn1} is the focus of the current manuscript, in the literature, also parabolic-elliptic versions are considered, in which the equation for $v$ is of elliptic type. For further discussion of the biological interpretation, see the surveys by Horstmann \cite{DH03,DH04}, Hillen and Painter \cite{HP09}, and Bellomo et al. \cite{BB15}, as well as works by Biler \cite{PB18} and Perthame \cite{BP04}.
	
	The mathematical literature on  Keller--Segel equations and their variants is vast, and we refer to the survey by Arumugam and Tyagi \cite{Arumugam_Tyagi} for a more complete review. To give a flavour of what kind of results are typically obtained, we discuss some selected works and also refer to the references therein.
	Whether solutions to \eqref{eqn1} exist globally is mainly determined by the Lyapunov functional
	\[
	\int \left(u \log (u) - uv +\frac{1}{2} |\nabla v|^2 +\frac{1}{2}v^2\right)\, \d x,
	\]
	taking all parameters in \eqref{eqn1} equal to one. 
	While in one spatial dimension the above is always lower bounded by Sobolev embeddings, in $d=2$  lower boundedness  requires smallness of the initial mass of  $u$. This may be shown, e.g., by the Moser--Trudinger inequality, and allows one to prove the global existence of solutions, as was done by Nagai, Senba and Yoshida in \cite{Nagai}. In the same two-dimensional setting, if the mass is large, the existence of blow-up solutions was shown by Herrero and Velázquez  \cite{herrero}. In higher dimensions, smallness in other spaces is required to ensure the global existence of solutions:  {Based on $L^pL^q$-estimates for the linearized equation, this was shown by Kozono and Sugiyama \cite{KOZONO20091}}  in scaling-critical spaces for $d\ge 3$.

	It is also natural to study Keller--Segel equations coupled to an advecting fluid velocity. In three  spatial dimensions, the existence and long-time behavior of the Keller--Segel--Navier--Stokes equation  was studied by Winkler \cite{Winkler2019Three}.
	Lorz  proved the existence of global in time solutions with small initial data in two and three dimensions for a coupled Keller--Segel--Stokes model \cite{Lorz2012Coupled}. 
	{He also provided numerical evidence that} the singularity formation in the Keller--Segel equation can be overcome by transport along the Stokes equation. {Such a prevention of blow-up was rigorously verified by Bedrossian and He \cite{Bedrossian_He} for transportation along a shear flow and  {we also refer to the work of Carrillo, Hittmeir and  J\"ungel} \cite{carillo_jungel} for other blow-up preventing mechanisms.}

	\subsection{Stochastic Keller--Segel equations}
	The macroscopic system \eqref{eqn1} is obtained as {effective dynamics of an underlying microscopic model}, 
	where individual organisms move randomly,  see, e.g., the work by Stevens \cite{AS00} and Fournier and Toma\v{s}evi\'{c} \cite{fournier2023particle}. 
	{Consequently,} on mesoscopic scales, a more accurate model may be obtained by incorporating the fluctuations around the macroscopic profile, leading to an additional noise term of Dean--Kawasaki type as shown by Huang and Qiu \cite{HJ2021}. {Other sources of noise in the Keller--Segel equations include transportation along a generic, temporally rough vector field leading to transport noise, or the inclusion of the effects of a random environment. The latter, modeled by the nonlinear noise term in \eqref{25.06.19.16.01}, is the focus of the current manuscript. }

	{In contrast to deterministic Keller--Segel equations, the mathematical literature in the stochastic case is more limited.}
	In one dimension, the existence and uniqueness of solutions to stochastic Keller--Segel equations were studied by  Hausenblas, Mukherjee and Tran \cite{HMT22a, HMT22b}. Misiats, Stanzhytskyi and Topaloglu \cite{MST22} and 
	Mayorcas and Toma\v{s}evi\'{c} \cite{Mayorcas2023Blowup} analyzed the influence of  noise on the blow-up behavior of solutions to the Keller--Segel on $\mathbb{R}^2$. 
	The global in time existence of renormalized kinetic solutions on the two-dimensional torus with a   Dean--Kawasaki noise term was recently obtained by Ji, Sun and Wu. 
	Martini and Mayorcas \cite{Martini2025Additive} studied paracontrolled solutions to the Keller--Segel equation in two dimensions with a linearized but more singular noise term of Dean--Kawasaki type.
	Flandoli, Galeati and Luo \cite{FGL21} demonstrated that blow-up in the Keller--Segel model can be delayed by transport noise. 
	Shang, Tian and Wang \cite{SJW19} investigated the asymptotic behavior of the solutions of non-autonomous stochastic Keller-Segel in one spatial dimension. The well-posedness of stochastic Keller--Segel--Navier--Stokes  and Keller--Segel equations in two  dimensions  was obtained by Zhang and Liu \cite{ZL25} 
	and Chen, Zhai and Zhang \cite{CZZ25}, respectively. 
	We also refer to the works by Braukhoff, Huber and J\"ungel \cite{BHJ24}, Dhariwal et al. \cite{DHJN21}, and Dhariwal, J\"ungel and Zamponi \cite{DJZ19} and Huber \cite{FH22} concerned with the mathematical analysis of other stochastic cross-diffusion equations.

	\subsection{Stochastic maximal $L^p$-regularity and evolution equations}
	The use of maximal regularity in the study of stochastic evolution equations was initiated by Da Prato in the 1980s \cite{da1985maximal}. While there $L^2$-estimates were considered, Da Prato studied $L^p$-estimates in time in the joint work with Lunardi \cite{DPL98}, see also his book with Zabczyk \cite{da2014stochastic}. The latter results on maximal estimates in real interpolation spaces were refined in Krylov's mixed $L^pL^q$-theory \cite{K96,K99}. A functional analytic approach for operators with a bounded $H^\infty$-calculus to such estimates was then obtained by van Neerven, Veraar and Weis  in \cite{NVW12a}, who also presented applications to evolution equations in \cite{NVW12b}. Applications to quasilinear SPDEs were first investigated by Hornung \cite{hornung2019quasilinear}. The latter was further developed by Agresti and Veraar in \cite{AV_SQEE_pt1,AV_SQEE_pt2} by considering time-weighted stochastic maximal regularity estimates in the spirit of the deterministic theory by Pr\"uss, Simonett and Wilke \cite{pruss2018critical}. The latter allowed them to prove parabolic smoothing for SPDEs like the Navier--Stokes equation with transport noise posed in critical spaces \cite{Agresti2024StochasticNS}. Other applications include the global in time well-posedness of primitive equations and the construction of smooth solutions to quasilinear SPDEs, in collaborations with Hieber, Hussein and  Saal \cite{Agresti2025StochasticPrimative} and the fourth author of the current manuscript \cite{agresti2025stochastic}, respectively. Said theory found applications to the stability analysis of travelling waves in stochastic reaction-diffusion equations in the work by van den Bosch and Hupkes \cite{bosch2024multidimensional}. For further information and selected open problems, we refer to the recent survey \cite{AV25}.

	\subsection{Organization of the remaining manuscript}
	As pointed out at the beginning of this introduction, in this manuscript, we explore the consequences of the theory of Agresti and Veraar \cite{AV_SQEE_pt1,AV_SQEE_pt2} for the well-posedness of the stochastic Keller--Segel equation \eqref{25.06.19.16.01}.
	For this, we proceed as follows: In the next section, we collect notation. In the preceding Section \ref{sec2}, we motivate and state our main results and their assumptions. Section \ref{sec3} is dedicated to the proof of our first main result (Theorem \ref{thm:local_wp}) on the local well-posedness of \eqref{25.06.19.16.01}. In Section \ref{sec4}, we give the  proof of the our second main result (Theorem \ref{thm_global_sol}) concerning the life time of solutions to  \eqref{25.06.19.16.01}  for small initial data.

	\section{Notation}\setcounter{equation}{0}
	
	\noindent
	\textbf{Function spaces.} Throughout, the space $\ell^2$ stands for the square-summable sequences, indexed by the natural numbers. More generally,  for a Banach space $X$ and a measure space $(S,\mu)$, we write $L^p(S,\mu;X)$ for the  Bochner-space of $X$-valued, $p$-integrable functions. If $S\subset \bR$, we write $C(S;X)$ for the space of continuous functions with values in $X$ equipped with the topology of uniform convergence on compacts. The subset of bounded functions, we denote by $C_b(S;X)$ and the class of functions $f\in C_b(S;X)$ with $\partial_s^lf\in C_b(S;X)$ for  $l\le k$ is denoted by $C^k_b(S;X)$. We denote the $d$-dimensional flat torus by $\bT^d$ and the set of $X$-valued periodic distributions, i.e., the linear and continuous mappings from $C^\infty(\bT^d)$ to $X$ by $\mathcal{D}'(\mathbb{T}^d;X)$. Then $H^{s,q}(\bT^d;X)$ denotes the $X$-valued Bessel potential space, consisting of those $f\in \mathcal{D}'(\mathbb{T}^d;X)$ satisfying 
	\begin{equation*}\label{definition bessel}
		\| f\|_{H^{s,q}(\mathbb{T}^d;X)}  = 
		\|(1 - \Delta)^{s/2} f\|_{L^q(\mathbb{T}^d;X)} < \infty, \end{equation*}
	for $s\in \bR$, $q\in (1,\infty)$. 
	Writing $X_\theta=[X_0,X_1]_{\theta} $ and $X_{\theta,p} = (X_0,X_1)_{\theta,p}$ for the complex and real interpolation spaces of a compatible couple $(X_0,X_1)$, we define  Besov spaces by
	$$B^s_{q,p}(\mathbb{T}^d)= (H^{\lceil s\rceil+1,q}(\mathbb{T}^d),H^{\lfloor s\rfloor-1,q}(\mathbb{T}^d))_{\theta,p},\quad 
	\theta = \frac{s- \lfloor s\rfloor+1 }{\lceil s\rceil - \lfloor s\rfloor +2 }
	,$$
	if $p\in (1,\infty)$. For $\kappa\in \bR$, we define weight functions  $w_\kappa^a(t):=|t-a|^\kappa$ and write $L^p((a,b),w_\kappa^a;X)$ for $L^p((a,b),w_\kappa^a\d t;X)$ and  $w_\kappa=w_\kappa^0$. Weighted Sobolev spaces are then introduced via 
	\[
	W^{k,p}((a,b), w^a_\kappa;X ) = \bigl\{
	f\in L^p((a,b),w_\kappa^a;X) \,\big|\, \partial_t^l f\in L^p((a,b),w_\kappa^a;X),\, l\le k
	\bigr\},
	\]
	where $\partial_t^lf$ denotes the weak derivative. The resulting complex interpolation spaces are denoted  by 
	\[
	H^{\theta,p}( (a,b), w^a_\kappa;X  )= [L^p((a,b),w_\kappa^a;X) , W^{1,p}((a,b), w^a_\kappa;X )]_\theta.
	\]
	Lastly, for a Hilbert space $\mathcal{H}$ with orthonormal basis $(\xi_n)_{n\geq 1}$, the space of $\gamma$-radonifying operators consist of those bounded linear $T \colon \mathcal{H}\to X$, such that $\sum_{n\ge 1} \gamma_n T\xi_n$ converges in $L^2(\Omega;X)$, in which case
	\begin{equation*}\label{gamma_rad}
		\|T\|_{\gamma(\mathcal{H},X)}^2 = \bE\biggl[\biggl\| \sum_{n\ge 1}\gamma_nT \xi_n\biggr\|_X^2\biggr]<\infty,
	\end{equation*}
	where $(\gamma_n)_{n\ge 1}$ is a family of independent standard normally distributed random variables. 
	
	\medskip
	\noindent
	\textbf{Probability.} Throughout, $(\Omega, \mathfrak{A}, \mathbb{P})$ is a fixed probability space with filtration   $\mathfrak{F}$. Moreover, we let 
	$W^i$, $i\in \{1,2\}$, be independent $\ell^2$-cylindrical $\mathfrak{F}$-Wiener processes, i.e., $(W^ie_k )_{k\in \bN,i\in \{1,2\} }$ is a family of  independent $\mathfrak{F}$-Brownian motions, where $e_k$ denotes the $k$-th unit vector in $\ell^2$.

	\section{Critical spaces,  results and discussion}\label{sec2}\setcounter{equation}{0}
	In this section, we derive scaling critical spaces for the Keller--Segel equation, state the definition of solutions and the assumptions, and present our main results.

	\subsection{Scaling critical  spaces}\label{sec: scaling critical spaces}
	In order to identify scaling critical spaces for \eqref{25.06.19.16.01}, we rescale space and time  according to the  heat operator, i.e., we set 
	
	\begin{equation*}\label{E8}
		u_{\lambda}(t,x):=\lambda^{h_1}u(\lambda t,\lambda^{1/2}x),\quad v_{\lambda}(t,x):=\lambda^{h_2}v(\lambda t,\lambda^{1/2}x),
	\end{equation*}
	for parameters $h_1$ and $h_2$ to be determined. Neglecting contributions of the noise terms, which we take to be of lower order, cf.\ Assumption \ref{Ass_noise} below, scaling criticality for the equation for $u$ requires that 
	\begin{align*}
		0 &= (\partial_tu_{\lambda}-\Delta u_{\lambda}+\mathrm{div}(u_{\lambda}\nabla v_{\lambda} ))(t,x)\\&=(\lambda^{1+h_1}\partial_tu-\lambda^{1+h_1}\Delta u+\lambda^{1+h_1+h_2}\mathrm{div}(u\nabla v))(\lambda t,\lambda^{1/2}x),
	\end{align*}
	for any solution $(u,v)$ to the Keller--Segel system, i.e., that $h_2=0$.  Requiring also scaling criticality for $v$ gives 
	\begin{align*} 0 =
		(\partial_tv_{\lambda}-\Delta v_{\lambda}+u_{\lambda})(t,x)=(\lambda^{1+h_2}\partial_tv-\lambda^{1+h_2}\Delta v+\lambda^{h_1}u)(\lambda t,\lambda^{1/2}x),
	\end{align*}
	where we find $h_1 = 1+ h_2$.
	In the critical case that $h_2=0$, this simply means  $h_1=1$. For the initial data, we have therefore identified the critical scaling relation  
	$$
	u_{0,\lambda}(x):=\lambda u_0(\lambda^{1/2}x),\ \text{ and } \ v_{0,\lambda}(x):=v_0(\lambda^{1/2}x),
	$$
	which is respected by the (homogeneous) Besov norms in
	\begin{align}\label{E_4}
		{B}_{q,p}^{d/q-2}:\ \|u_{0,\lambda}\|_{{B}_{q,p}^{d/q-2}}&\simeq \lambda^{1+\frac{1}{2}\left(\frac{d}{q}-2-\frac{d}{q}\right)}\|u_0\|_{{B}_{q,p}^{d/q-2}}=\|u_0\|_{{B}_{q,p}^{d/q-2}},
	\end{align}
	and
	\begin{align}\label{E_5}
		{B}_{q_v,p}^{d/q_v}:\ \|v_{0,\lambda}\|_{{B}_{q_v,p}^{d/q_v}}&\simeq \lambda^{\frac{1}{2}\left(\frac{d}{q_v}-\frac{d}{q_v}\right)}\|v_0\|_{{B}_{q_v,p}^{d/q_v}}=\|v_0\|_{{B}_{q_v,p}^{d/q_v}}.
	\end{align}
	The latter scale naturally emerges from the trace theory for stochastic maximal $L^p_\kappa$-theory in Bessel potential spaces as we see in the following section, cf.~\cite{ALV}. 
	{The scaling critical spaces are then precisely the ones which require sharp estimates on the Keller--Segel nonlinearity to prove local well-posedness by means of \cite{AV_SQEE_pt1,AV_SQEE_pt2}, see Section \ref{sec3} below.}

	\subsection{The Keller--Segel equation as a stochastic evolution equation}\label{sec: Keller-segel as evolution equation}
	To make use of the framework \cite{AV_SQEE_pt1,AV_SQEE_pt2}, we consider the equation \eqref{25.06.19.16.01}--\eqref{25.06.19.16.02} using stochastic maximal $L^p_\kappa$-estimates in
	\begin{equation*}\label{eq: definition of X0, X1}
		(X_0 , X_1) \,=\, (H^{s,q}(\bT^d)\times H^{s_v,q_v} (\bT^d) ,  H^{2+s,q}(\bT^d)\times H^{2+s_v,q_v}(\bT^d) ),
	\end{equation*}
	with the usual parameter restriction that 
	{
		\begin{equation}\label{Eq_standard}
			s,s_v \in \bR, \quad  q,q_v \in [2,\infty), \, p\in (2,\infty),\, \kappa \in [0,p/2-1) \quad\text{or}\quad q=q_v=p=2, \,\kappa =0 ,
		\end{equation}
	}
	from \cite[Assumption 3.1]{AV_SQEE_pt1}. 
	Following \cite{schmeisser}, the resulting complex  interpolation spaces are 
	\begin{equation}\label{eq: definition of X beta}
		X_{\beta}\,=\,[X_0,X_1]_{\beta}\,=\,H^{2\beta+s,q}(\mathbb{T}^d)\times H^{2\beta+s_v,q_v}(\mathbb{T}^d),\qquad \beta\in (0,1)  ,
	\end{equation}
	while real interpolation yields
	\[
	X_{\beta,p}\,=\, (X_0,X_1)_{\beta,p} \,=\, B_{q,p}^{2\beta+s}(\mathbb{T}^d)\times B_{q_v,p}^{2\beta+s_v}(\mathbb{T}^d),\qquad \beta\in (0,1),\;p\in (1,\infty).
	\]
	To ensure scaling criticality of the resulting `trace space' 
	\begin{equation}\label{t_space}
		X_{1 -  (1+\kappa) /p,  p} \,=\,
		B^{s +2 - 2(1+\kappa)/p}_{q,p}(\bT^d)\times  B^{s_v +2 - 2(1+\kappa)/p}_{q_v,p}(\bT^d),
	\end{equation}
	and to facilitate our proof of local well-posedness, we impose that
	\begin{align}
		\label{Eq_pc1_new}
		&s\in (-2,-1], \; s_v \in (-1,1],
		\\    
		\label{Eq_pc2_new}
		&s_v + 2 - 2(1+\kappa)/p -d/q_v \,=\, 0 ,
		\\
		\label{Eq_pc3_new}
		&s + 2 -d/q \,=\, s_v -d/q_v ,\\
		\label{Eq_pc4_new}
		&s \ge -d/q , \,  s_v \ge  -d/q_v ,\\
		\label{Eq_pc5_new}
		&s_v\le s+2.
	\end{align}
	Indeed, \eqref{Eq_pc2_new} and \eqref{Eq_pc3_new} imply that the leading order part of the norm of \eqref{t_space} is precisely the scaling critical quantity \eqref{E_4}--\eqref{E_5}. 
	The remaining assumptions are used in our estimates on the nonlinearities in \eqref{25.06.19.16.01}, cf.\ Section \ref{sec3} below.
	\begin{rem}[Discussion of the parameter asssumptions \eqref{Eq_pc1_new}--\eqref{Eq_pc5_new}]\label{rmk: examples of choices of parameters} \,
		
		\begin{enumerate}[(i)]
			\item Firstly, let us point our that the restrictions $s>-2$ and $s_v \in (-1,1]$ from \eqref{Eq_pc1_new} are in fact implied by the remaining conditions. Indeed, combining the upper bound $s\le -1$ from \eqref{Eq_pc1_new} with \eqref{Eq_pc5_new}, immediately limits $s_v\le 1$. On the other hand, by adding $s_v+d/q_v$ to both sides of \eqref{Eq_pc2_new}, and dividing by 2, we obtain that 
			\begin{align*}
				s_v + 1 = (1+\kappa)/p + (s_v+d/q_v)/2.
			\end{align*}
			The right hand side of this is guaranteed to be positive by \eqref{Eq_pc4_new}, and thus $s_v>-1$. Similarly, by inserting \eqref{Eq_pc3_new} in \eqref{Eq_pc2_new}, adding to both $s+d/q$ and dividing by $2$, we get that
			\begin{align}\label{e:1}
				s + 2 = (1+\kappa)/p + (s+d/q)/2,
			\end{align}
			which is positive by \eqref{Eq_pc4_new} as well. The latter enforces then $s>-2$.
			\item The conditions \eqref{Eq_pc1_new}--\eqref{Eq_pc5_new} restrict  us to work in dimension $d\ge 3$: Indeed, 
			rearranging \eqref{e:1}, we find that
			\[
			d/(2q) \,=\, s/2+2-(1+\kappa)/p \,\ge\,-d/(2q) +2-(1+\kappa)/p ,
			\] 
			after inserting \eqref{Eq_pc4_new}. In other words, we impose that
			\begin{equation}\label{e:2}
				d/2\,\ge\, 
				d/q \,\ge\, 2-(1+\kappa)/p\,\ge \, 3/2,
			\end{equation}
			where the first and last estimate are due to \eqref{Eq_standard}. In other words,  in dimensions $d\in \{1,2\}$ the above conditions \eqref{Eq_standard} and \eqref{Eq_pc1_new}--\eqref{Eq_pc5_new} would  not be compatible.        
			\item Even in the case $d=3$ the chain of inequalities \eqref{e:2} requires that $p=2$ and $\kappa=0$, which in turn enforces that $q=q_v=2$, again by \eqref{Eq_standard}. The conditions \eqref{Eq_pc2_new}--\eqref{Eq_pc3_new} then identify $s_v = 1/2$ and $s=-3/2$. Thereby, the only possible choice of critical spaces in dimension $d=3$ is given by 
			\[
			\bigl(p,\kappa,(s,q),(s_v,q_v)\bigr) \,=\, \bigl(2,0,(-3/2,2),(1/2,2)\bigr),
			\]
			leading to the scaling critical trace space
			\[
			B_{2,2}^{-1/2}(\mathbb{T}^d)\times B_{2,2}^{3/2}(\mathbb{T}^d)=H^{-1/2,2}(\bT^d)\times H^{3/2,2}(\bT^d).
			\]
			In dimension $d=4$, there is already more flexibility and for example 
			\begin{equation*}
				\begin{bmatrix}
					p & \kappa  & s & q  &s_v&q_v \\
					2 & 0  &-1 & 2 &   1 & 2 \\
					4 & 3/5  & - 8/5 &5/2&  0 & 10/3\\
					4 & 3/5  & - 8/5 &5/2&  -1/5 & 4  
				\end{bmatrix}
			\end{equation*}
			are three distinct choices in line with \eqref{Eq_standard} and \eqref{Eq_pc1_new}--\eqref{Eq_pc5_new}.
		\end{enumerate}
	\end{rem}
	
	In order to treat \eqref{25.06.19.16.01}--\eqref{25.06.19.16.02} as a stochastic evolution equation, we define the linear block operator
	\begin{equation}
		\label{def:A}
		-A:=\begin{pmatrix}
			\Delta&0\\
			\mathrm{Id}&\Delta-\mathrm{Id}
		\end{pmatrix},
	\end{equation}
	acting on the pair $U = (u,v)$. 
	Moreover, we denote the nonlinearity arising from  cross diffusion by
	\begin{equation}
		\label{Eq_F1}
		F(u,v)\,:=\,(-\mathrm{div}(u\nabla v),0).
	\end{equation}
	We further interpret  the stochastic terms as the operators
	\begin{equation}\label{def:G}
		G(u,v)(e_n, 0) \,:=\, (g_n^1(u,v,\nabla v),0),\ \text{ and } \ G(u,v) (0, e_m ) \,:=\, (0,g^2_m(v)),
	\end{equation}
	acting on $\mathcal{H}:=\ell^2 \times \ell^2$, 
	where $(e_n)_{n\geq1}$ is the standard orthonormal basis of $\ell^2$.
	
	With this notation at hand, \eqref{25.06.19.16.01}--\eqref{25.06.19.16.02} may be reformulated as the stochastic evolution equation
	\begin{equation}
		\label{25.06.19.16.54}
		\begin{cases}
			\mathrm{d}U+AU\,\mathrm{d}t=F(U)\,\mathrm{d}t+G(U)\,\mathrm{d}W,\\
			U(0)=(u_0,v_0),
		\end{cases}
	\end{equation}
	where $U=(u,v)$ and $W = (W^1, W^2)$ is an $\mathcal{H}$-cylindrical $\mathfrak{F}$-Wiener process. 
	
	\subsection{Definition of solutions}
	To give a rigorous meaning to \eqref{25.06.19.16.54}, we need to capture also stochastically integrable $X_{1/2}$-valued processes, which naturally leads to  $\gamma$-radonifying operators, more precisely, the class $\gamma(\mathcal{H},X_{1/2})$. We note that by the $\gamma$-Fubini isomorphism \cite[Theorem 9.4.8]{hytonen2018analysis} we have 
	\begin{align*}
		\gamma(\mathcal{H},X_{1/2}) &= \gamma(\ell^2\times \ell^2,H^{1+s,q}(\bT^d) ) \times \gamma(\ell^2\times \ell^2,H^{1+s_v,q_v}(\bT^d) ) \\& = H^{1+s,q}(\ell^2\times \ell^2)\times H^{1+s_v,q_v}(\ell^2\times \ell^2).
	\end{align*}
	With this at hand, we recast the solution concept from \cite[Definitions 4.3--4.4]{AV_SQEE_pt1} to our setting, where in the third point we make use of the subtle adjustment observed in  \cite[Remark 5.6]{AV_SQEE_pt2}.
	
	\begin{defn}[Solutions to \eqref{25.06.19.16.01}]\label{def2.2}
		Assume that the parameters $s,s_v,q,q_v,p,\kappa$ satisfy \eqref{Eq_standard}. 
		Let also $\sigma:\Omega\to[0,\infty]$ be a stopping time, and $(u,v):[0,\sigma)\times\Omega\to H^{2+s,q}(\mathbb{T}^d)\times H^{2+s_v,q_v}(\mathbb{T}^d)$ be an $\mathfrak{F}$-adapted process.
		\begin{enumerate}[(i)]
			\item The pair $((u,v),\sigma)$ is called an \emph{$L^p_{\kappa}$-strong solution} to the stochastic Keller--Segel system \eqref{25.06.19.16.01} if, $\bP$-a.s.,
			\begin{align*}
				u\in L^{p}\left((0,\sigma),w_\kappa;H^{2+s,q}(\mathbb{T}^d)\right)\cap C\left([0,\sigma];B^{s+2-{2(1+\kappa)}/{p}}_{q,p}(\mathbb{T}^d)\right),& \\
				\label{eqn_reg2}
				v\in L^{p}\left((0,\sigma),w_{\kappa}; H^{2+s_v,q_v}(\mathbb{T}^d)\right)\cap C\left([0,\sigma]; B^{s+2-{2(1+\kappa)}/{p}}_{q_v,p}(\mathbb{T}^d)\right), &\\
				\mathrm{div}(u\nabla v)\in L^1\left((0,\sigma);H^{s,q}(\mathbb{T}^d)\right),\nonumber&\\\nonumber
				g^1(u,v,\nabla v) \in  L^2((0,\sigma);H^{1+s,q}(\bT^d;\ell^2)),\quad 
				g^2(v) \in  L^2((0,\sigma);H^{1+s_v,q_v}(\bT^d ;\ell^2)),&
			\end{align*}
			and 
			\begin{align*}
				&u_t-u_0-\int_0^t\Delta u\,\d s=-\int_0^t\mathrm{div}(u\nabla v)\,\mathrm{d}s+\sum_{n\geq 1}\int_0^t g_n^{1}(u,v ,\nabla v)\,\mathrm{d}W^1_n,
			\end{align*}
			as well as
			\begin{equation*}   
				v_t-v_0-\int_0^t\Delta v+u-v\,\mathrm{d}s=\sum_{n\geq 1}\int_0^t g^2_n(v)\,\mathrm{d}W_n^2 ,
			\end{equation*}for all $t\in[0,\sigma]$ in $ H^{s,q}(\mathbb{T}^d)$ and $H^{s_v,q_v}(\mathbb{T}^d)$, respectively.
			\item \label{item_localizing_Seq} The pair $ ((u,v), \sigma) $ is called an \emph{$ L^p_\kappa $-local solution} to \eqref{25.06.19.16.01} if there exists stopping times  $ \sigma_n \uparrow \sigma $, $\bP$-a.s., such that $ ((u,v), \sigma_n) $ is an $ L^p_\kappa $-strong solution to \eqref{25.06.19.16.01} for each $ n $.

			\item An \emph{$ L^p_\kappa $-local solution}  $ ((u,v), \sigma) $ to \eqref{25.06.19.16.01} is called \emph{maximal} if for any other  \emph{$ L^p_\kappa $-local solution} solution $ ((\tilde{u},\tilde{v}), \tau) $ with $(\tilde u_0, \tilde v_0)=(u_0,v_0)$, we have $ \tau \leq \sigma $, $\bP$-a.s., and
			$$(\tilde{u}_t,\tilde{v}_t) = 
			(u_t,v_t)   \quad \text{for all } t \in [0, \tau), \text{ $\bP$-almost surely.}
			$$
		\end{enumerate}    
	\end{defn}

	\subsection{Assumptions and statement of the main results}\label{sec: results}
	
	For the local well-posedness we allow for the following choice of functions $g^i$, $i\in \{1,2\}$. 
	
	\begin{assumption}[Noise coefficients: local well-posedness]
		\label{Ass_noise}
		\begin{enumerate}[(i)]
			\item 
			\label{Ai}
			Assume that $g^1\colon \bR^{2+d} \to \ell^2$ is Lipschitz
			\begin{equation}\label{eqB}
				\|g^1(z) - g^1(\bar{z})\|_{\ell^2}\,\lesssim\, |z-\bar{z}|, \  \text{ for every }\ z,\bar{z}\in\mathbb{R}^{2+d}.
			\end{equation}
			
			\item  \label{Aii}
			Assume that $g^2\colon \bR \to \ell^2$ is a function satisfying $(g^2)' \in C^2_b(\mathbb R;\ell^2)$.
		\end{enumerate}
	\end{assumption}
	
	The reason for the significantly stronger restriction on $g^2$ than $g^1$ in the above is that  $v$ lies in a space of higher regularity than $u$, making nonlinear estimates more involved, see Lemma \ref{lem: bound for g1} and Lemmas \ref{lemma: estimate for g^2 if sv is in -1 0}--\ref{lemma: estimate for g^2 if sv is in 0 1} below.
	Let us state our first result on the local well-posedness of \eqref{25.06.19.16.01} in critical spaces.
	\begin{thm}[Local well-posedness]\label{thm:local_wp}
		Assume that the parameters $s,s_v,q,q_v,p,\kappa$ satisfy \eqref{Eq_standard} and  \eqref{Eq_pc1_new}--\eqref{Eq_pc5_new}, and let $g^{i}$, $i\in \{1,2\}$, satisfy Assumption \ref{Ass_noise}. 
		Then, for any $\mathfrak{F}_0$-measurable 
		\begin{equation}
			\label{eq_IV}
			(u_0,v_0)\colon \Omega\to B^{d/q-2}_{q,p}(\bT^d)\times  B^{d/q_v}_{q_v,p}(\bT^d),
		\end{equation}
		there exists an   $L^p_\kappa$-maximal solution $((u,v),\sigma)$ to \eqref{25.06.19.16.01} starting at \eqref{eq_IV} and satisfying $\sigma>0$, $\bP$-almost surely.
		{Moreover, if $((u,v),\sigma)$ and $((\bar{u},\bar{v}),\bar{\sigma})$ are $L^p_\kappa$-maximal solutions to \eqref{25.06.19.16.01}, then 
			\begin{align*}
				\sigma = \bar{\sigma}\qquad \text{and}\qquad 
				(u,v)|_{[0,\sigma)} = (\bar{u},\bar{v})|_{[0,\sigma)}, 
			\end{align*}
			a.s., on $\{{(u_0,v_0)=(\bar{u}_0,\bar{v}_0)}\}\in \mathfrak{F}_0$, and 
			for $p\ne 2$ we have the additional regularity 
			\begin{equation*}
				u \in  H^{\theta, p}((0,\sigma_n),w_\kappa; H^{2(1-\theta)+s,q}(\mathbb{T}^d) )
				\quad \text{and}\quad v\in  H^{\theta, p}((0,\sigma_n),w_\kappa; H^{2(1-\theta)+s_v,q_v}(\mathbb{T}^d)),
			\end{equation*}
			for each $\theta\in [0,1/2)$ and $ \sigma_n \uparrow \sigma $ as in Definition \ref{def2.2}~\eqref{item_localizing_Seq}.
		}
	\end{thm}
	
	To obtain a lower bound on the maximal existence time $\sigma$ for small initial data, we need to ensure that the noise nonlinearities are small for small values of $(u,v)$. 
	More precisely, we demand that 
	\begin{align}
		\label{assumption_for_global}
		g^{1}(0,0,0) = 0 \quad \text{and}\quad g^2(0) = 0.
	\end{align}
	The resulting growth bounds allow us to iterate the stochastic maximal regularity estimates repeatedly to obtain our second main result.
	\begin{thm}[Small data implies long time of existence, with high probability] \label{thm_global_sol}
		Assume that the parameters $s,s_v,q,q_v,p,\kappa$ satisfy \eqref{Eq_standard} and  \eqref{Eq_pc1_new}--\eqref{Eq_pc5_new}, and let $g^{i}$, $i\in \{1,2\}$, satisfy Assumptions \ref{Ass_noise} and \eqref{assumption_for_global}. 
		Moreover, let $(u_0,v_0)$ as in \eqref{eq_IV} be $\mathfrak{F}_0$-measurable  and $((u,v),\sigma)$ be the $L^p_\kappa$-maximal solution  to \eqref{25.06.19.16.01} provided by Theorem \ref{thm:local_wp}. 
		{Then, for each $T <\infty$, there exists a constant $C_{T}>0$ such that for each constant $\varepsilon \in (0,1)$, if}
		\[
		\bE \|(u_0,v_0) \|_{B_{q,p}^{d/q-2}(\mathbb{T}^d)\times B_{q_v,p}^{d/q_v}(\mathbb{T}^d)}^p \,\le\, C_{T} \varepsilon,
		\]
		then we have that 
		\[
		\bP( \sigma \ge T) \ge 1-\varepsilon.
		\]
		Moreover, there exists a stopping time $\tau\in (0, \sigma]$ such that $\bP(\tau \ge T) \ge 1 - \varepsilon$ and, if $p>2$, it holds that
		\begin{equation}
			\label{26.04.16.14.29}
			\begin{aligned}
				&\bE\left[ \mathbf{1}_{\tau \ge T} \|(u,v)\|^p_{H^{\theta, p}((0,T),w_\kappa; H^{2(1-\theta)+s,q}(\mathbb{T}^d)\times H^{2(1-\theta)+s_v,q_v}(\mathbb{T}^d))} \right] \\&
				\lesssim_\theta \bE \|(u_0,v_0)\|_{ B_{q,p}^{d/q-2}(\mathbb{T}^d)\times B_{q_v,p}^{d/q_v}(\mathbb{T}^d)}^p,
			\end{aligned}
		\end{equation}
		for all $\theta \in [0,1/2)$. 
		If $p=2$, the previous estimate \eqref{26.04.16.14.29} has to be replaced by 
		\begin{equation*}
			\begin{aligned}
				&\bE\left[ \mathbf{1}_{\tau \ge T} (\|(u,v)\|^2_{C([0,T]; H^{1+s,2}(\mathbb{T}^d)\times H^{1+s_v,2}(\mathbb{T}^d))}+\|(u,v)\|^2_{L^2([0,T],w_\kappa; H^{2+s,2}(\mathbb{T}^d)\times H^{2+s_v,2}(\mathbb{T}^d))}) \right]\\&
				\lesssim \bE \|(u_0,v_0)\|_{H^{d/2-2,2}(\mathbb{T}^d)\times H^{d/2,2}(\mathbb{T}^d)}^2.
			\end{aligned}
		\end{equation*}
	\end{thm}

	\section{Proof of local well-posedness}\setcounter{equation}{0}\label{sec3}
	In this section, we prove the local well-posedness of the stochastic Keller--Segel system \eqref{25.06.19.16.01}, i.e., Theorem \ref{thm:local_wp}.
	To do this, we first prove stochastic $L^p_{\kappa}$-maximal regularity estimates for the linear part of \eqref{25.06.19.16.01} in Subsection \ref{SS_lin}, and then prove bounds on the deterministic nonlinearity and the noise term in Subsections \ref{sec: Estimates on the deterministic nonlinearity} and \ref{sec: Estimates on the noise term}, respectively. We conclude in Subsection \ref{SS_lwp_conclusion}.
	\subsection{Stochastic maximal regularity for the linear part}\label{SS_lin}
	We recall from Subection \ref{sec: Keller-segel as evolution equation} that the equation \eqref{25.06.19.16.01} can be formulated as the evolution equation \eqref{25.06.19.16.54}, the linear part of which takes the form
	\begin{equation}
		\label{25.06.19.17.46}
		\begin{cases}
			\mathrm{d}U+AU\,\mathrm{d}t\,=\,f\,\mathrm{d}t+g\,\mathrm{d}W,\quad t\in(a,T),\\
			U(a)=0,
		\end{cases}
	\end{equation}
	for given inhomogeneities 
	$f$ and $g$.
	We first define what it means for the linear part $A$ to have stochastic maximal $L^p_{\kappa}$-regularity, cf.\ \cite[Definitions 3.4 and 3.5]{AV_SQEE_pt1}. For this purpose,   we momentarily consider two general type $2$ UMD-Banach spaces $X_1\hookrightarrow X_0$ and assume that either 
	\begin{itemize}
		\item $p\in (2,\infty) $ and $\kappa\in [0,p/2-1)$, or
		\item $p=2$, $\kappa=0$ and $X_0$ and $X_1$ are Hilbert spaces.
	\end{itemize}
	We denote also in this case $X_{\beta}\,=\,[X_0,X_1]_{\beta}$.
	\begin{defn}\label{def: definition of SMR}
		Let $X_0$ and $X_1$ be as above, and $A\colon X_1 \to X_0$ be linear and bounded.
		\begin{enumerate}[(i)]
			\item $A$ belongs to the class $\mathcal{SMR}_{p,\kappa}$ if the following holds: 
			For every $T\in(0,\infty)$, there exists $C_T>0$ such that for all intervals $[a,b]\subseteq [0,T]$, and for every pair of progressively measurable processes
			\begin{equation}
				\label{25.06.21.17.31}
				(f,g)\in L^p(\Omega;L^p((a,b),w_{\kappa}^a;X_0))\times L^p(\Omega;L^p((a,b),w_{\kappa}^a;\gamma(\mathcal{H};X_{1/2}))),
			\end{equation}
			there exists a unique $L_\kappa^p$-strong solution to \eqref{25.06.19.17.46}, i.e., a progressively measurable
			\begin{align*}\label{eqn264}
				U\in L^2((a,b)\times \Omega, w_\kappa^a;X_1), \qquad \text{s.t.}\quad 
				U_t +  \int_a^t AU\, \d s  = \int_a^t f\,\d s +\int_0^t
				g\,\d W ,
			\end{align*}for all $t\in [a,b]$, a.s., 
			and the latter satisfies the bound
			\begin{equation*}
				\label{25.06.25.13.17}
				\mathbb{E}\|U\|_{L^p((a,b),w_{\kappa}^a;X_1)}^p\leq C_T \mathbb{E}\|f\|_{L^p((a,b),w_{\kappa}^a;X_0)}^p+C_T\mathbb{E}\|g\|_{L^p((a,b),w_{\kappa}^a;\gamma(\mathcal{H},X_{1/2}))}^p.
			\end{equation*}
			
			\item For $p\in(2,\infty)$, we say that  $A$ belongs to $\mathcal{SMR}_{p,\kappa}^{\bullet}$ if
			\begin{itemize}
				\item $A\in\mathcal{SMR}_{p,\kappa}$,
				\item for all $\theta\in(0,1/2)$ and $T\in(0,\infty)$, there exists a constant $C_{T,\theta}>0$ such that for all intervals $[a,b]\subseteq [0,T]$, and for every pair of progressively measurable processes $(f,g)$ in \eqref{25.06.21.17.31},
				the $L_\kappa^p$-strong solution $U$ to \eqref{25.06.19.17.46} $U$  satisfies
				$$
				\mathbb{E}\|U\|_{H^{\theta,p}((a,b),w_{\kappa}^a;X_{1-\theta})}^p\leq C_{T,\theta} \mathbb{E}\|f\|_{L^p((a,b),w_{\kappa}^a;X_0)}^p+C_{T,\theta}\mathbb{E}\|g\|_{L^p((a,b),w_{\kappa}^a;\gamma(\mathcal{H},X_{1/2}))}^p.
				$$
			\end{itemize}
			\item We say that an operator $A$ belongs to $\mathcal{SMR}_{2,0}^{\bullet}$ if
			\begin{itemize}
				\item $A\in \mathcal{SMR}_{2,0}$,
				\item for all $T\in(0,\infty)$, there exists a constant $C_{T}>0$ such that for all intervals $[a,b]\subseteq [0,T]$, and for every pair of progressively measurable processes $(f,g)$ in \eqref{25.06.21.17.31},
				the $L^2$-strong solution $U$ to \eqref{25.06.19.17.46} satisfies
				$$
				\mathbb{E}\|U\|_{C([a,b];X_{1/2})}^2\leq C_{T} \mathbb{E}\|f\|_{L^2((a,b);X_0)}^2+C_{T}\mathbb{E}\|g\|_{L^2((a,b);\gamma(\mathcal{H},X_{1/2}))}^2.
				$$
			\end{itemize}
		\end{enumerate}
	\end{defn}
	We show that linear part of \eqref{25.06.19.16.01}, i.e., \eqref{def:A} belongs to $\mathcal{SMR}_{p, \kappa}^{\bullet}$.
	\begin{thm}\label{thm: A B in SMR}
		Assume \eqref{Eq_standard} and \eqref{Eq_pc1_new}--\eqref{Eq_pc5_new}, then the operator defined by
		$$
		-A := \begin{pmatrix}
			\Delta & 0 \\
			\mathrm{Id} & \Delta - \mathrm{Id}
		\end{pmatrix} \colon H^{s+2, q}(\mathbb{T}^d) \times H^{s_v+2, q_v}(\mathbb{T}^d) \to 
		H^{s, q}(\mathbb{T}^d) \times H^{s_v, q_v}(\mathbb{T}^d)
		$$
		belongs to $\mathcal{SMR}_{p, \kappa}^{\bullet}$.
	\end{thm}
	\begin{proof}
		To prove the assertion, we use the transference result (see, {e.g.}, \cite[Proposition 3.8]{AV_SQEE_pt1}).
		It suffices to verify the following two conditions:
		\begin{enumerate}[(1)]
			\item There exists a reference operator $\tilde{A} \in \mathcal{SMR}_{p,\kappa}^{\bullet}$.
			\item The operator $A$ belongs to $\mathcal{SMR}_{p,\kappa}$.
		\end{enumerate}
		
		\vspace{2mm}
		\noindent
		\emph{Step 1. Construction of the reference operator.}
		Let us consider the  operator
		$$
		-\tilde{A} := \begin{pmatrix}
			\Delta & 0 \\
			0 & \Delta
		\end{pmatrix} \colon H^{s+2, q}(\mathbb{T}^d) \times H^{s_v+2, q_v}(\mathbb{T}^d) \to 
		H^{s, q}(\mathbb{T}^d) \times H^{s_v, q_v}(\mathbb{T}^d).
		$$
		Since 
		\begin{equation}\label{eqn323}-\Delta \colon H^{s+2, q}(\mathbb{T}^d)  \to H^{s, q}(\mathbb{T}^d) \end{equation} 
		lies in $\mathcal{SMR}_{p, \kappa}^{\bullet}$ and likewise 
		\begin{equation}\label{eqn324}-\Delta \colon H^{s_v+2, q_v}(\mathbb{T}^d)  \to H^{s_v, q_v}(\mathbb{T}^d) \end{equation} 
		by \cite[Lemma 5.2]{AV24}, we find that $\tilde{A}\in \mathcal{SMR}_{p, \kappa}^{\bullet} $ due to  the diagonal structure. 
		
		\vspace{2mm}
		\noindent
		\emph{Step 2. Verification of $A\in\mathcal{SMR}_{p,\kappa}$.}
		We aim to show the existence, uniqueness and estimates of solutions to \eqref{25.06.19.17.46}. The key observation is that for $U=(u,v)$, $f = (f_u, f_v)$, and $g = (g_u, g_v)$, the system attains the following triangular form:
		\begin{align}
			\mathrm{d}u - \Delta u\,\mathrm{d}t &= f_u\,\mathrm{d}t + g_u\,\mathrm{d}W, \label{eq: U1} \\
			\mathrm{d}v + ((\mathrm{Id}-\Delta)v - u)\,\mathrm{d}t &= f_v\,\mathrm{d}t + g_v\,\mathrm{d}W. \label{eq: U2}
		\end{align}
		Since \eqref{eqn323} lies in $ \mathcal{SMR}_{p, \kappa}^{\bullet}$, the first equation \eqref{eq: U1} admits a unique $L^p_{\kappa}$-strong solution $u$. 
		Moreover, stochastic maximal regularity  yields
		\begin{equation}
			\label{26.02.02.15.24}
			\begin{aligned}
				&\mathbb{E}\|u\|_{L^p((a, b), w_{\kappa}^a; H^{s+2, q}(\mathbb{T}^d))}^p\\
				&\leq C_T \Big( \mathbb{E}\|f_u\|_{L^p((a, b), w_{\kappa}^a ; H^{s, q}(\mathbb{T}^d))}^p + \mathbb{E}\|g_u\|_{L^p((a, b), w_{\kappa}^a; \gamma(\mathcal{H}; H^{s+1, q}(\mathbb{T}^d)))}^p \Big)\\ &
				\leq C_T \Big( \mathbb{E}\|f\|_{L^p((a, b), w_{\kappa}^a; X_0)}^p + \mathbb{E}\|g\|_{L^p((a, b), w_{\kappa}^a; \gamma(\mathcal{H}; X_{1/2}))}^p \Big).
			\end{aligned}
		\end{equation}
		
		To find the solution $v$ for the second equation \eqref{eq: U2}, we consider an auxiliary problem.
		First, we observe that due to \eqref{Eq_pc3_new} and \eqref{Eq_pc5_new}, we have the sharp Sobolev embedding $H^{s+2, q}(\mathbb{T}^d) \hookrightarrow H^{s_v, q_v}(\mathbb{T}^d)$. 
		Consequently, the solution $u$ to \eqref{eq: U1} satisfies
		\begin{equation}\label{eqn435}
			\mathbb{E}\|u\|_{L^p((a, b), w_{\kappa}^a; H^{s_v, q_v}(\mathbb{T}^d) )}^p \leq C \mathbb{E}\|u\|_{L^p((a, b), w_{\kappa}^a; H^{s+2, q}(\mathbb{T}^d) )}^p.
		\end{equation}
		This implies that the term $\mathrm{e}^{t}(f_v + u)$ 
		belongs to $L^p((a,b),w_{\kappa}^a;H^{s_v, q_v}(\mathbb{T}^d) ))$, and therefore serves as a well-defined forcing term.
		Since \eqref{eqn324} lies in $\mathcal{SMR}_{p,\kappa}^{\bullet}$, there exists a unique $L^p_{\kappa}$-strong solution  $\tilde{v} \in L^p((a, b), w_{\kappa}^a; H^{s_v+2, q_v}(\mathbb{T}^d) ))$ to the following equation:
		\begin{equation}\label{eq: U2 tilde}
			\mathrm{d}\tilde{v} - \Delta \tilde{v}\,\mathrm{d}t = \mathrm{e}^{t}(f_v + u)\,\mathrm{d}t + \mathrm{e}^{t}g_v\,\mathrm{d}W, \quad \tilde{v}(a) = 0.
		\end{equation}
		Moreover, $\tilde{v}$ satisfies the stochastic maximal regularity estimate
		\begin{align*}
			&\mathbb{E}\|\tilde{v}\|_{L^p((a, b), w_{\kappa}^a; H^{s_v+2, q_v}(\mathbb{T}^d) ))}^p \\&
			\leq C_T \left( \mathbb{E}\|\mathrm{e}^{t}(f_v + u)\|_{L^p((a, b), w_{\kappa}^a; H^{s_v, q_v}(\mathbb{T}^d) ))}^p + \mathbb{E}\|\mathrm{e}^{t}g_v\|_{L^p((a, b), w_{\kappa}^a; \gamma(\mathcal{H}; H^{s_v+1, q_v}(\mathbb{T}^d)))}^p \right) \\&
			\leq C_T \left( \mathbb{E}\|f_v\|_{L^p((a, b), w_{\kappa}^a;H^{s_v, q_v}(\mathbb{T}^d) ))}^p + \mathbb{E}\|u\|_{L^p((a, b), w_{\kappa}^a; H^{s_v, q_v}(\mathbb{T}^d) ))}^p \right)
			\\& \quad + C_T\mathbb{E}\|g_v\|_{L^p((a, b), w_{\kappa}^a; \gamma(\mathcal{H}; H^{s_v+1, q_v}(\mathbb{T}^d) )))}^p .
		\end{align*}
		Substituting the estimates on $u$ from \eqref{26.02.02.15.24} and \eqref{eqn435} into the above inequality, we obtain
		\begin{align*}
			\mathbb{E}\|\tilde{v}\|_{L^p((a, b), w_{\kappa}^a;H^{s_v+2, q_v}(\mathbb{T}^d))}^p
			\leq C_T\left(  \mathbb{E}\|f\|_{L^p((a, b), w_{\kappa}^a; X_0)}^p + \mathbb{E}\|g\|_{L^p((a, b), w_{\kappa}^a; \gamma(\mathcal{H}; X_{1/2}))}^p \right).
		\end{align*}
		We now define $v(t) := \mathrm{e}^{-t}\tilde{v}(t)$. 
		By  It\^o's produce rule,  $v$ is indeed an $L^p_{\kappa}$-strong solution to the original equation \eqref{eq: U2} and satisfies the required estimates. For uniqueness, one may start from an $L^p_\kappa$-strong solution $v$ to \eqref{eq: U2} and obtain the unique $L^p_\kappa$-strong solution to \eqref{eq: U2 tilde} by the inverse relation $\tilde{v}(t) = e^t v(t)$. 
		Combining the results for $u$ and $v$, we conclude that $A \in \mathcal{SMR}_{p, \kappa}$.
	\end{proof}

	\subsection{Estimates on the deterministic nonlinearity}\label{sec: Estimates on the deterministic nonlinearity}
	We estimate the deterministic non-linearity $F$ defined in \eqref{Eq_F1} in the space $X_0$. This amounts to the following lemma.
	\begin{lem}\label{lem: bound for F1}
		Assume  \eqref{Eq_standard} and  \eqref{Eq_pc1_new}--\eqref{Eq_pc5_new}, then we have the bounds
		\begin{align}\begin{split}\label{Eq101}&\|
				\mathrm{div}(u\nabla v) -\mathrm{div}(\bar{u} \nabla \bar{v}) \|_{H^{s,q} (\bT^d) } \\& \lesssim  \| v\|_{H^{{\gamma_v },q_v}(\bT^d)} \|u- \bar{u}\|_{H^{\gamma,q}(\bT^d)} +\|\bar{u}\|_{H^{\gamma,q}(\bT^d)} \|v- \bar{v}\|_{H^{\gamma_v,q_v}(\bT^d)},
			\end{split}
		\end{align}
		and 
		\begin{align}\begin{split}\label{Eq1011}&\|
				\mathrm{div}(u\nabla v)  \|_{H^{s,q} (\bT^d) } \,\lesssim \,  \|u\|_{H^{\gamma,q}(\bT^d)}\| v\|_{H^{\gamma_v,q_v}(\bT^d)},
			\end{split}
		\end{align}
		for all $(u,v),(\bar{u},\bar{v})\in H^{\gamma,q}(\mathbb{T}^d)\times H^{\gamma_v,q_v}(\mathbb{T}^d)$, where $ \gamma= s+2-(1+\kappa)/p$, and $\gamma_v$ is defined analogously with $s$ replaced by $s_v$.
	\end{lem}
	\begin{proof}
		From the triangle inequality, we obtain that 
		\begin{align}\begin{split}\label{Eq100}
				&
				\| \mathrm{div}(u\nabla v) -\mathrm{div}(\bar{u} \nabla \bar{v}) \|_{H^{s,q} (\bT^d) } \\&
				\lesssim 
				\|u\nabla v-\bar{u}\nabla{\bar{v}}\|_{H^{1+s,q}(\mathbb{T}^d;\bR^d)}\\& \le  \|(u-\bar{u})\nabla v\|_{H^{1+s,q}(\mathbb{T}^d;\bR^d)}+\|\bar{u}\nabla( v-\bar{v})\|_{H^{1+s,q}(\mathbb{T}^d;\bR^d)}.
			\end{split}
		\end{align}
		Regarding the Sobolev index of the latter space, we observe that since $s\le -1$ by \eqref{Eq_pc1_new}, it follows that 
		$$s+1 -d/q \le -d/q.$$
		At the same time, the fact that $s \ge -d/q$ by \eqref{Eq_pc4_new} implies 
		$$s+1 -d/q \ge 1-2d/q >-d,$$
		since also $q\ge 2$ by \eqref{Eq_standard}.
		All in all, we have that the Sobolev index of $H^{1+s,q}(\mathbb{T}^d;\bR^d)$ satisfies
		\[
		s+1 - d/q \,\in \, (-d,-d/q], 
		\]
		and therefore there exists $R \in (1,q]$, such that the embedding $L^R(\bT^d)\hookrightarrow H^{1+s,q}(\mathbb{T}^d) $
		is sharp. The latter is characterized by 
		\begin{equation}\label{Eq_R}
			-d/R =s+1 -d/q.
		\end{equation}
		With this choice, we can process \eqref{Eq100} to the estimate 
		\begin{equation}
			\|\mathrm{div}(u\nabla v) -\mathrm{div}(\bar{u} \nabla \bar{v})  \|_{H^{s,q} (\bT^d;\bR^d) } \lesssim\|(u-\bar{u})\nabla v\|_{L^R(\mathbb{T}^d;\bR^d)}+\|\bar{u}\nabla( v-\bar{v})\|_{L^R(\mathbb{T}^d;\bR^d)}.
		\end{equation}
		To show the desired estimate \eqref{Eq101}, it suffices to bound the former of the two terms on the right hand side appropriately. Indeed, while this results in the former term on the right hand side of \eqref{Eq101}, a repetition of the following chain of arguments bounds the second term by the second term on the right hand side of \eqref{Eq101}.

		To do so,  we set $r_v  =  \frac{d}{1-(1+\kappa)/p} \in (d, 2d]$, cf.\ \eqref{Eq_standard},  and define $ r  $  via $1/R = 1/r_v +1/r$. Using \eqref{Eq_R}, this can be expressed via 
		\begin{equation}\label{Eq1}
			-d/r = s+2 -(1+\kappa)/p -d/q ,
		\end{equation}
		and to ensure that $r\in (R,\infty)$, we convince ourselves that $R<r_v$. Indeed, the latter is equivalent to $-d/R < -d/r_v$, which by \eqref{Eq_R} amounts to 
		\[
		s+1 -d/q < (1+\kappa)/p \,-\, 1\qquad \text{or equivalently}\qquad  s+2 - (1+\kappa)/p-d/q <0.
		\]
		But since \eqref{Eq_standard} and \eqref{Eq_pc2_new}--\eqref{Eq_pc3_new} imply
		\[
		s+2 - (1+\kappa)/p-d/q =(1+\kappa)/p - 2 < 0,
		\]
		we obtain that $r\in (R,\infty)$. 
		Thereby, with this choice, we have by H\"older's inequality
		\[
		\|(u-\bar{u})\nabla v\|_{L^R(\mathbb{T}^d;\bR^d)} \le  \|\nabla v\|_{ L^{r_v}(\mathbb{T}^d;\bR^d)}\|u - \bar{u}\|_{L^r (\mathbb{T}^d)} \lesssim  \| v\|_{ H^{1, r_v}(\mathbb{T}^d)} \|u - \bar{u}\|_{L^r (\mathbb{T}^d)}.
		\]
		
		To prove the claim, we use the Sobolev embedding to estimate the above terms. To this end, we observe that the Sobolev index of $L^r(\bT^d)$ and $H^{\gamma,q}(\bT^d)$ coincide, due to \eqref{Eq1}. Similarly, since 
		\begin{equation}\label{Eq11}
			1  - d/r_v = (1+\kappa)/p=  s_v + 2 - (1+\kappa)/p -d/q_v ,
		\end{equation}
		by \eqref{Eq_pc2_new}, also the Sobolev indices of $H^{1,r_v}(\bT^d)$ and $H^{\gamma_v,q_v}(\bT^d)$ are equal. Consequently, the desired estimate
		\begin{equation*}
			\| v\|_{H^{1, r_v}(\mathbb{T}^d)} \|u - \bar{u}\|_{L^r (\mathbb{T}^d)} \,\lesssim\,
			\| v\|_{H^{\gamma_v,q_v}(\bT^d)} \|u- \bar{u}\|_{H^{\gamma,q}(\bT^d)} ,
		\end{equation*}
		follows as soon as we convince ourselves that 
		\begin{equation}\label{Eq99}s+2 -(1+\kappa)/p \,\ge\, 0 \qquad \text{and}\qquad 
			s_v + 2 - (1+\kappa)/p \ge 1.
		\end{equation}
		For the former, we use that
		\[
		s+4 - 2(1+\kappa)/p -d/q = 0,
		\]
		which follows from \eqref{Eq_pc2_new}--\eqref{Eq_pc3_new}, and that therefore
		\[
		s+2 -(1+\kappa)/p  = (s+d/q)/2.
		\]
		That the latter is non-negative is imposed in \eqref{Eq_pc4_new}. Similarly, starting from \eqref{Eq_pc2_new}, we deduce that 
		\[
		s_v + 1 - (1+\kappa)/p  = (s_v+d/q_v)/2,
		\]
		so that the latter inequality of \eqref{Eq99} is ensured by \eqref{Eq_pc4_new} too.
		This completes the proof of the first claim \eqref{Eq101}.
		The second claim \eqref{Eq1011}
		follows  by setting $(\bar{u},\bar{v}) =(0,0)$ in \eqref{Eq101} and using that then $\mathrm{div}(\bar{u} \nabla \bar{v})=0$.
	\end{proof}
	\subsection{Estimates on the noise term}\label{sec: Estimates on the noise term}
	In this section, we prove estimates on the non-linear noise term $G$ in the space $X_{1/2}$. We begin with the following estimate on $g^1$.
	\begin{lem}\label{lem: bound for g1}
		Assume \eqref{Eq_standard} and  \eqref{Eq_pc1_new}--\eqref{Eq_pc5_new} and that $g^1$ satisfies Assumption \ref{Ass_noise}~\eqref{Ai}.
		Then we have the bounds
		\begin{align}\begin{split}\label{Eq1010}&\|
				g^1(u,v,\nabla v ) \,-\, g^1(\bar{u},\bar{v},\nabla \bar{v}) \|_{H^{1+s,q} (\bT^d;\ell^2) } \lesssim_{{g^1}} 
				\|(u-\bar{u},v- \bar{v})\|_{H^{\gamma,q}(\bT^d)\times H^{\gamma_v,q_v}(\bT^d)},
			\end{split}
		\end{align}
		and 
		\begin{align}\begin{split}\label{Eq10101}\|
				g^1(u,v,\nabla v )  \|_{H^{1+s,q} (\bT^d;\ell^2) } \,&\lesssim_{g^1} \, 1 + \|{(u,v)}\|_{{H^{\gamma,q}(\bT^d)\times H^{\gamma_v,q_v}(\bT^d)}} ,
			\end{split}
		\end{align}
		for all $(u,v),(\bar{u},\bar{v})\in H^{\gamma,q}(\mathbb{T}^d)\times H^{\gamma_v,q_v}(\mathbb{T}^d)$, where $ \gamma= s+2-(1+\kappa)/p$, and $\gamma_v$ is defined analogously with $s$ replaced by $s_v$.
	\end{lem}
	\begin{proof}
		As a first step towards \eqref{Eq1010}, we observe that as a consequence of the Lipschitz condition \eqref{eqB} and the vector-valued Sobolev embedding 
		\begin{equation*}
			L^R(\bT^d;\ell^2)\hookrightarrow H^{1+s,q}(\mathbb{T}^d; \ell^2) 
		\end{equation*}
		with exponents as in \eqref{Eq_R}, we have
		\begin{align*}\begin{split}\label{eqC}
				&\|
				g^1(u,v,\nabla v ) - g^1(\bar{u},\bar{v},\nabla \bar{v}) \|_{H^{1+s,q} (\bT^d;\ell^2) }
				\\ & \lesssim 
				\|
				g^1(u,v,\nabla v ) - g^1(\bar{u},\bar{v},\nabla \bar{v}) \|_{L^R (\bT^d;\ell^2) }
				\\& \lesssim_{g^1} 
				\|
				|u -\bar{u}| + |v-\bar{v}| + |\nabla v -\nabla \bar{v} | \|_{L^R (\bT^d) }.
			\end{split}
		\end{align*}
		Then, by the triangle inequality, it suffices to bound the summands of 
		\begin{align*}&
			\|u -\bar{u}\|_{L^R (\bT^d) }+ \|v -\bar{v}\|_{L^R (\bT^d) } + \|\nabla v - \nabla \bar{v}\|_{L^R (\bT^d) }
			\\& \lesssim  \|u -\bar{u}\|_{L^r (\bT^d) } + \|v -\bar{v}\|_{H^{1,{r_v}} (\bT^d) } ,
		\end{align*}
		where we take $r$ and $r_v$ as in the proof of Lemma \ref{lem: bound for F1} above. But then the Sobolev embeddings 
		\begin{align*}
			H^{\gamma,q}(\bT^d)
			\, \hookrightarrow \, & L^r (\bT^d),
			\qquad
			H^{\gamma_v,q_v}(\bT^d) 
			\,\hookrightarrow\   H^{1,{r_v}}(\bT^d),
		\end{align*}
		verified in \eqref{Eq1}, \eqref{Eq11} and \eqref{Eq99}, yield the first estimate \eqref{Eq1010}.
		The growth bound \eqref{Eq10101} is then a consequence of the simple estimate
		\[
		\|     g^1(u,v,\nabla v )  \|_{H^{1+s,q}  (\bT^d;\ell^2) } \lesssim 
		\|     g^1(u,v,\nabla v ) - g^1(0,0,0)  \|_{H^{1+s,q}  (\bT^d;\ell^2) } 
		+ \|g^1(0,0,0)\|_{\ell^2},
		\]
		and \eqref{Eq1010}.
	\end{proof}

	The estimate for $g^2$ is split up into the cases that $s_v\in (-1,0]$ and the case that $s_v\in(0,1]$.
	We observe that estimating $g^2$ in the less regular space when $s_v\in (-1,0]$ provides a better bound compared to estimating $g^2$ in the higher regularity space corresponding to $s_v\in(0,1]$. This is essentially due to additional products arising from repeated differentiation of nonlinearities.
	\begin{lem}\label{lemma: estimate for g^2 if sv is in -1 0}
		Assume \eqref{Eq_standard} and  \eqref{Eq_pc1_new}--\eqref{Eq_pc5_new} and that $g^2$ satisfies Assumption \ref{Ass_noise}~\eqref{Aii}.
		If $s_v \in (-1,0]$, then 
		\begin{equation} \label{eq: desired bound for g2 in the case sv is in 0,1}\|g^2(v) - g^2(\bar v)\|_{H^{1+s_v,q_v}(\mathbb{T}^d;\ell^2)} \lesssim_{g^2} \left(1+\|v\|_{H^{\gamma_v,q_v}(\mathbb{T}^d)}+\|\bar{v}\|_{H^{\gamma_v,q_v}(\bT^d)}\right)\|v-\bar{v}\|_{H^{\gamma_v,q_v}(\bT^d)},
		\end{equation}
		and 
		\begin{equation}
			\label{eq: desired bound for g2 growth}
			\|g^2(v)\|_{H^{1+s_v,q_v}(\mathbb{T}^d;\ell^2)} \lesssim_{g^2} \left(1+\|v\|_{H^{\gamma_v,q_v}(\mathbb{T}^d)}\right)\|v\|_{H^{\gamma_v,q_v}(\bT^d)},
		\end{equation}
		for all $v,\bar{v}\in H^{\gamma_v, q_v}(\mathbb{T}^d)$, where $\gamma_v=s_v+2 -\frac{1+\kappa}{p}$.
	\end{lem}
	\begin{proof}
		We start by proving the local Lipschitz estimate.
		Since $g^2$ is differentiable, we may evaluate
		\begin{equation*}
			g^2(v)-g^2(\bar{v})=(v-\bar{v})\int_0^1 (g^2)'(\bar{v}+\theta (v-\bar{v}))\,\mathrm{d}\theta.
		\end{equation*}
		Using this and the paraproduct estimate 
		\begin{equation*}
			\|fh\|_{H^{1+s_v,q_v}(\mathbb{T}^d;\ell^2)}\lesssim \|f\|_{H^{1+s_v,q_v}(\mathbb{T}^d)}\|h\|_{L^\infty(\mathbb{T}^d;\ell^2)}+\|f\|_{L^{\infty}(\mathbb{T}^d)}\|h\|_{H^{1+s_v,q_v}(\mathbb{T}^d;\ell^2)},
		\end{equation*}
		from \cite[Proposition 4.1]{AV24}, we obtain that
		\begin{align}\label{eq: equation for I1 I2}&
			\| g^2(v)-g^2(\bar{v}) \|_{H^{1+s_v,q_v}(\bT^d;\ell^2)}\nonumber\\
			& \lesssim \|v-\bar{v}\|_{H^{1+s_v,q_v}(\bT^d)} \biggl\| \int_0^1 (g^2)'(\bar{v}+\theta (v-\bar{v}))\,\mathrm{d}\theta\biggr\|_{L^\infty(\bT^d;\ell^2) } \nonumber\\
			&\quad +
			\|v-\bar{v}\|_{L^\infty(\bT^d)} \biggl\| \int_0^1 (g^2)'(\bar{v}+\theta (v-\bar{v}))\,\mathrm{d}\theta\biggr\|_{H^{1+s_v,q_v}(\bT^d;\ell^2) }\nonumber\\
			&=:I_1+I_2.
		\end{align}
		By Assumption \ref{Ass_noise}~\eqref{Aii}  and  the fact that $1+s_v\leq \gamma_v$, it follows that
		\begin{equation*}
			I_1\lesssim_{g^2} \|v-\bar{v}\|_{H^{\gamma_v,q_v}(\bT^d)},
		\end{equation*}
		which we observe appears on the right hand side of the desired estimate \eqref{eq: desired bound for g2 in the case sv is in 0,1}. 
		
		To complete the proof it remains to handle $I_2$ for which we first consider the integral term.
		For brevity, we denote  $v_{\theta}:=\bar{v}+\theta(v-\bar{v})$.
		Recalling that $s_v\in(-1,0]$, it follows 
		\begin{align*}\begin{split}
				\biggl\| \int_0^1 (g^2)'(\bar{v}+\theta (v-\bar{v}))\,\mathrm{d}\theta\biggr\|_{H^{1+s_v,q_v}(\bT^d;\ell^2)}\leq  &\int_0^1 \left\|(g^2)'(v_\theta)\right\|_{H^{1+s_v,q_v}(\bT^d;\ell^2) }\,\mathrm{d}\theta\\
				\lesssim &\int_0^1 \left\|(g^2)'(v_\theta)\right\|_{H^{1,q_v}(\bT^d;\ell^2) }\,\mathrm{d}\theta.
				\noindent
			\end{split}
		\end{align*}
		By the definition of the $H^{1,q_v}(\mathbb{T}^d;\ell^2)$-norm and Assumption \ref{Ass_noise}~\eqref{Aii}, the right hand side of the above can be controlled by
		\begin{align*}
			\begin{split}&
				\int_0^1 \left(\left\|(g^2)'(v_\theta)\right\|_{L^{q_v}(\bT^d;\ell^2) }+\left\|\nabla (g^2)'(v_\theta)\right\|_{L^{q_v}(\bT^d;\ell^2(\bR^d)) }\right)\,\mathrm{d}\theta
				\le  C_{g^2} + K_1,
			\end{split}
		\end{align*}
		with  $$K_1:=\int_0^1\left\|\nabla (g^2)'(v_\theta)\right\|_{L^{q_v}(\bT^d;\ell^2(\bR^d)) }\,\mathrm{d}\theta.$$
		The chain rule $\nabla (g^2)'(v_\theta) = (g^2)''(v_\theta ) \nabla v_\theta $ then gives 
		\begin{equation}\label{eqn44}
			\left\|\nabla (g^2)'(v_\theta) \right\|_{{L^{q_v}(\bT^d;\ell^2(\bR^d)) }}\,\lesssim_{{g^2}} \| v_\theta \|_{H^{1, q_v}(\bT^d)} \le \| v \|_{H^{1, q_v}(\bT^d)} + \| \bar{v} \|_{H^{1, q_v}(\bT^d)},
		\end{equation}
		and therefore 
		\begin{align*}
			\begin{split}
				K_1\lesssim_{{g^2} }  \| v \|_{H^{1, q_v}(\bT^d)} + \| \bar{v} \|_{H^{1, q_v}(\bT^d)}.
			\end{split}
		\end{align*}
		
		\vspace{2mm}
		\noindent
		Consequently, we showed that $I_2$ satisfies the bound
		\begin{equation*}
			I_2 \lesssim_{g^2}  \|v-\bar{v}\|_{L^\infty(\bT^d)}(1+ \|v\|_{H^{1,q_v}(\bT^d) }+\|\bar{v}\|_{H^{1,q_v}(\bT^d) }).
		\end{equation*}
		By \eqref{Eq_pc2_new}, the first term in the product on the right hand side can be further bounded by the Sobolev embedding
		\begin{equation*}
			H^{\gamma_v,q_v}(\bT^d) \hookrightarrow L^\infty(\mathbb{T}^d),
		\end{equation*}
		which yields
		\begin{equation*}
			\|v-\bar{v}\|_{L^\infty(\bT^d)}\leq \|v-\bar{v}\|_{H^{\gamma_v,q_v}(\bT^d)}.
		\end{equation*}
		Using once more that $1+s_v\leq \gamma_v$ for the remaining terms and combining with the estimate for $I_1$  gives the local Lipschitz estimate \eqref{eq: desired bound for g2 in the case sv is in 0,1}.
		
		Together with 
		\[
		\|
		g^2(v )  \|_{H^{1+s_v,q_v} (\bT^d;\ell^2) } \,\le\, \|
		g^2(v) -  g^2(0 ) \|_{H^{1+s_v,q_v} (\bT^d;\ell^2) } \,+\, \|g^2(0)\|_{\ell^2},
		\]
		the growth bound \eqref{eq: desired bound for g2 growth} follows as well.
	\end{proof}
	{
		We now prove an estimate for $g^2$ in the case that $s_v\in(0,1]$.
		\begin{lem}\label{lemma: estimate for g^2 if sv is in 0 1}
			Assume \eqref{Eq_standard} and  \eqref{Eq_pc1_new}--\eqref{Eq_pc5_new} and that $g^2$ satisfies Assumption \ref{Ass_noise}~\eqref{Aii}.
			If $s_v \in (0,1]$, then for every $\eta >0$
			\begin{align}\begin{split}\label{eqn_LL_g2_case2}
					&\|g^2(v) - g^2(\bar v)\|_{H^{1+s_v,q_v}(\mathbb{T}^d;\ell^2)} \\
					& \lesssim_{(\eta ,g^2)}\Big(1 + \|v\|_{H^{2,q_v}(\mathbb{T}^d)}+\|\bar{v}\|_{H^{2,q_v}(\bT^d)} \Big)  \|v - \bar v\|_{H^{\zeta_\eta,q_v}(\mathbb{T}^d)}
					\\&
					\quad + \left(\|v\|_{H^{1+\frac{d}{2q_v},q_v}(\mathbb{T}^d)}^{1+s_v}+\|\bar{v}\|_{H^{1+\frac{d}{2q_v},q_v}(\bT^d)}^{1+s_v}\right)\|v - \bar v\|_{H^{\zeta_\eta,q_v}(\mathbb{T}^d)},
				\end{split}
			\end{align}
			and 
			\begin{align}\begin{split}\label{eqn_GB_g2_case2}
					& \|g^2(v)\|_{H^{1+s_v,q_v}(\mathbb{T}^d;\ell^2)} \\
					& \lesssim_{(\eta  ,g^2)}\left(1 + \|v\|_{H^{2,q_v}(\mathbb{T}^d)}\right) \|v\|_{H^{\zeta_\eta,q_v}(\mathbb{T}^d)}
					+ \|v\|_{H^{1+\frac{d}{2q_v},q_v}(\mathbb{T}^d)}^{1+s_v}\|v\|_{H^{\zeta_\eta,q_v}(\mathbb{T}^d)},
				\end{split}
			\end{align}
			for all $v,\bar{v}\in H^{\max\{2, \zeta_\eta\}, q_v}(\mathbb{T}^d)$, where $\zeta_\eta = s_v+2-2\frac{1+\kappa}{p}+\eta$.
		\end{lem}
		\begin{proof}
			We start as in the proof of  Lemma \ref{lemma: estimate for g^2 if sv is in -1 0}
			from the estimate \eqref{eq: equation for I1 I2}. 
			Since \eqref{Eq_pc2_new} ensures the Sobolev embedding 
			\begin{align}
				\label{eqn_Sob_Emb}
				H^{s_v+2-2\frac{1+\kappa}{p}+\eta,q_v}(\mathbb{T}^d)\hookrightarrow L^\infty(\mathbb{T}^d),
			\end{align}
			we can bound already
			\begin{equation}\label{eqn9723432}
				I_1 \lesssim_{(g^2,\eta)} \|v-\bar{v}\|_{H^{\zeta_\eta,q_v}(\bT^d)},
			\end{equation}
			where $I_1$ is defined in \eqref{eq: equation for I1 I2}.
			For $I_2$ (defined in \eqref{eq: equation for I1 I2}), we use that 
			$$
			H^{1+s_v,q_v}(\bT^d;\ell^2) = \big[H^{1,q_v}(\bT^d;\ell^2),H^{2,q_v}(\bT^d;\ell^2)\big]_{s_v-1},
			$$
			to estimate the appearing integral by
			\begin{align}\label{eq: interpolation product in the bound for g2}\nonumber
				&  \biggl\| \int_0^1 (g^2)'(v_\theta))\,\mathrm{d}\theta\biggr\|_{H^{1+s_v,q_v}(\bT^d;\ell^2)}
				\\&\leq \int_0^1 \left\|(g^2)'(v_\theta)\right\|_{H^{1,q_v}(\bT^d;\ell^2)}^{1-s_v} \left\|(g^2)'(v_\theta)\right\|_{H^{2,q_v}(\bT^d;\ell^2) }^{s_v}\,\mathrm{d}\theta. 
			\end{align}
			Then, on the one hand, we have by  \eqref{eqn44} the estimate
			\begin{equation}\label{eqn45}
				\left\|(g^2)'(v_\theta)\right\|_{H^{1,q_v}(\bT^d;\ell^2)} \lesssim_{g^2} 1+\|v_\theta \|_{H^{1,q_v}(\bT^d)}.
			\end{equation}
			For the $H^{2,q_v}(\bT^d;\ell^2)$-norm, on the other hand,  we use that 
			\begin{align*}
				\begin{split}&
					\left\|(g^2)'(v_\theta)\right\|_{H^{2,q_v}(\bT^d;\ell^2) }\lesssim 
					\left\|(g^2)'(v_\theta)\right\|_{L^{q_v}(\bT^d;\ell^2) }+\left\|\Delta(g^2)'(v_\theta)\right\|_{L^{q_v}(\bT^d;\ell^2) } 
					\le  C_{g^2} + K_2,    
				\end{split}
			\end{align*}
			so that it remains to control $K_2$. 
			By the chain rule and product rule, we have
			\begin{equation*}
				\Delta(g^2)'(v_\theta)=\nabla\cdot((g^2)''(v_\theta)\nabla v_\theta)=(g^2)''(v_\theta)\Delta v_\theta + (g^2)'''(v_\theta)|\nabla v_\theta|^2,
			\end{equation*}
			and by the triangle inequality, we can bound the $L^{q_v}(\bT^d;\ell^2)$-norms of the two terms on the right hand side separately.
			For the first we simply have 
			\begin{equation*}
				\|(g^2)''(v_\theta)\Delta v_\theta\|_{L^{q_v}(\bT^d;\ell^2) }\lesssim_{{g^2}} \|v_\theta\|_{H^{2,q_v}(\mathbb{T}^d)},
			\end{equation*}
			in light of Assumption \ref{Ass_noise}~\eqref{Aii}. For the second term, we bound instead
			\begin{align*}\label{eq: bounding the gradient square of v theta}\nonumber
				\|(g^2)'''(v_\theta)|\nabla v_\theta|^2\|_{L^{q_v}(\bT^d;\ell^2) }&\lesssim_{{g^2}} \||\nabla v_\theta|^2\|_{L^{q_v}(\bT^d) }\le \|\nabla v_\theta\|^2_{L^{2q_v}(\bT^d) }\\& \le \|v_\theta\|^2_{H^{1,2q_v}(\bT^d) }
				\leq \|v_\theta\|_{H^{1+\frac{d}{2q_v},q_v}(\bT^d)}^2,
			\end{align*}
			where in the final inequality, we used the sharp Sobolev embedding 
			\begin{equation*}
				H^{1+\frac{d}{2q_v},q_v}(\bT^d) \hookrightarrow H^{1,2q_v}(\bT^d).
			\end{equation*}
			Combining with the above, we proved that 
			\begin{equation*}
				\left\|(g^2)'(v_\theta)\right\|_{H^{2,q_v}(\bT^d;\ell^2) }\leq C_{g^2}+ K_2 \lesssim_{g^2} 1+\|v_\theta\|_{H^{2,q_v}(\mathbb{T}^d)}+\|v_\theta\|_{H^{1+\frac{d}{2q_v},q_v}(\bT^d)}^2.
			\end{equation*}
			Inserting this and \eqref{eqn45} in \eqref{eq: interpolation product in the bound for g2}, results in the bound 
			\begin{align*}\label{eq: final bound for the product in the estimate for g2}\nonumber
				&  \int_0^1  \left\|(g^2)'(v_\theta)\right\|_{H^{1,q_v}(\bT^d;\ell^2)}^{1-s_v} \left\|(g^2)'(v_\theta)\right\|_{H^{2,q_v}(\bT^d;\ell^2) }^{s_v}\,\mathrm{d}\theta\\& \nonumber
				\lesssim_{g^2} \int_0^1 \left(1+ \|v_\theta\|_{H^{1,q_v}(\bT^d)}\right)^{1-s_v} \left(1+\|v_\theta\|_{H^{2,q_v}(\mathbb{T}^d)}+\|v_\theta\|_{H^{1+\frac{d}{2q_v},q_v}(\bT^d)}^2\right)^{s_v}\,\mathrm{d}\theta\\& \nonumber
				\lesssim_{g^2} \int_0^1 1+\|v_\theta\|_{H^{2,q_v}(\mathbb{T}^d)}+\|v_\theta\|_{H^{1+\frac{d}{2q_v},q_v}(\bT^d)}^{1+s_v}\,\mathrm{d}\theta
				\\&
				\le 1 + \|v\|_{H^{2,q_v}(\bT^d)} +
				\|\bar{v}\|_{H^{2,q_v}(\bT^d)} +\|{v}\|_{H^{1+\frac{d}{2q_v},q_v}(\bT^d)}^{1+s_v} +\|\bar{v}\|_{H^{1+\frac{d}{2q_v},q_v}(\bT^d)}^{1+s_v}
				.
			\end{align*}
			Finally, using once more the embedding \eqref{eqn_Sob_Emb}, we deduce 
			\begin{align*}
				I_2 \lesssim_{(g^2,\eta)}&
				\| v-\bar{v}\|_{ H^{\zeta_\eta,q_v}(\mathbb{T}^d)}\\&\times\Bigl(1 + \|v\|_{H^{2,q_v}(\bT^d)} +
				\|\bar{v}\|_{H^{2,q_v}(\bT^d)} +\|{v}\|_{H^{1+\frac{d}{2q_v},q_v}(\bT^d)}^{1+s_v} +\|\bar{v}\|_{H^{1+\frac{d}{2q_v},q_v}(\bT^d)}^{1+s_v}\Bigr).
			\end{align*}
			Together with \eqref{eqn9723432}, this entails the local Lipschitz estimate \eqref{eqn_LL_g2_case2}. As in the preceding lemma, the growth bound \eqref{eqn_GB_g2_case2} is a consequence thereof.
		\end{proof}
		
		\subsection{Proof of local well-posedness}\label{SS_lwp_conclusion}
		Using the results from the previous subsections, we are able to prove local well-posedness of the stochastic Keller--Segel system \eqref{25.06.19.16.01} by means of \cite[Theorem 4.8 (c)]{AV_SQEE_pt1}. The idea of the latter result is that a contraction estimate for a suitably truncated  version of the stochastic evolution equation \eqref{25.06.19.16.54} may be established, for instance, if the following  conditions are met:
		\begin{enumerate}[(1)]
			\item \label{I1} $A\in \mathcal{SMR}_{p,\kappa}^{\bullet}$.
			\item \label{I2} \cite[Hypothesis (HF)]{AV_SQEE_pt1}: There exist  parameters
			\begin{equation}\label{eq: parameter conditions for local wp}
				\varphi_j\in \Big(1-\frac{1+\kappa}{p},1\Big),\quad \beta_j\in\Big(1-\frac{1+\kappa}{p},\varphi_j\Big] ,\quad  \rho_j\geq 0,
			\end{equation}
			for $j\in \{1,\dots, m_F\}$ and  $C,L<\infty$, such that
			\begin{align}\label{eq: Lipschitz type estimate for F in local wp}
				\|F(U)-F(\bar{U})\|_{X_0}&\leq L\sum_{j=1}^{m_F}(1+\|U\|^{\rho_j}_{X_{\varphi_j}}+\|\bar{U}\|^{\rho_j}_{X_{\varphi_j}})\|U-\bar{U}\|_{X_{\beta_j}},\\
				\|F(U)\|_{X_0}&\leq C\sum_{j=1}^{m_F}(1+\|U\|^{\rho_j}_{X_{\varphi_j}})\|U\|_{X_{\beta_j}}+C,\label{eq: growth condition estimate for F in local wp}\end{align}
			for all $U,\bar{U}\in X_1$.
			It is crucial that the parameters from \eqref{eq: parameter conditions for local wp}   obey the (sub-)criticality condition
			\begin{equation}\label{eq: subcriticality condition in local wp proof}
				\rho_j\left(\varphi_j-1+\frac{1+\kappa}{p}\right)+\beta_j\leq 1.
			\end{equation}
			
			\item\label{I3} \cite[Hypothesis (HG)]{AV_SQEE_pt2}: There exist parameters \eqref{eq: parameter conditions for local wp} for $j\in \{m_F+1,\dots ,m_F+m_G\}$ satisfying the (sub-) criticality condition \eqref{eq: subcriticality condition in local wp proof}, such that 
			\begin{align}\label{eq: Lipschitz type estimate for G in local wp}
				\|G(U)-G(\bar{U})\|_{ \gamma(\mathcal{H},X_{1/2})}&\leq L\sum_{j=m_F+1}^{m_F+m_G}(1+\|U\|^{\rho_j}_{X_{\varphi_j}}+\|\bar{U}\|^{\rho_j}_{X_{\varphi_j}})\|U-\bar{U}\|_{X_{\beta_j}},
				\\\label{eq: growth condition estimate for G in local wp}
				\|G(U)\|_{ \gamma(\mathcal{H},X_{1/2})}&\leq C\sum_{j=m_F+1}^{m_F+m_G}(1+\|U\|^{\rho_j}_{X_{\varphi_j}})\|U\|_{X_{\beta_j}}+C,
			\end{align}
			for all $U,\bar{U}\in X_1$, up to enlarging $C$ and $L$.
		\end{enumerate} 
		\begin{rem}
			Under the conditions \eqref{I1}--\eqref{I3} above, \cite[Theorem 4.8 (c)]{AV_SQEE_pt1} not only yields the assertions of Theorem \ref{thm:local_wp} for \eqref{25.06.19.16.01}, but also continuous dependence on the initial data may be deduced. We refer the interested reader to the cited result.
		\end{rem}
		\begin{proof}[Proof of Theorem \ref{thm:local_wp}]
			We establish  Theorem \ref{thm:local_wp} by verifying the conditions \eqref{I1}--\eqref{I3} for $A$, $F$, and $G$ as defined in \eqref{def:A}, \eqref{Eq_F1} and \eqref{def:G}.

			\begin{enumerate}
				\item The fact that $A\in \mathcal{SMR}_{p,\kappa}^{\bullet}$ is the content of Theorem \ref{thm: A B in SMR}.
				\item  
				The two estimates \eqref{eq: Lipschitz type estimate for F in local wp} and \eqref{eq: growth condition estimate for F in local wp} are the ones  verified in Lemma \ref{lem: bound for F1}, if we set
				\begin{equation}\label{eqn43}
					m_F=1, \quad  \varphi_1=\beta_1=1-\frac{1+\kappa}{2p},\quad  \rho_1=1.
				\end{equation}
				Moreover, that $\varphi_1=1-\frac{1+\kappa}{2p}\in(1-\frac{1+\kappa}{p},1)$ follows from \eqref{Eq_standard}, while the (sub-)criticality condition \eqref{eq: subcriticality condition in local wp proof} holds by
				\begin{equation*}
					1\left(1-\frac{1+\kappa}{2p}-1+\frac{1+\kappa}{p}\right)+1-\frac{1+\kappa}{2p}=1.
				\end{equation*}
				We remark that the equality in the above reflects the scaling criticality of the spaces we work in, as discussed in Subsection \ref{sec: scaling critical spaces}.
				\item {
					We claim that also the estimates from Lemmas \ref{lem: bound for g1} and \ref{lemma: estimate for g^2 if sv is in -1 0} and \ref{lemma: estimate for g^2 if sv is in 0 1} can be cast in the form
					\eqref{eq: Lipschitz type estimate for G in local wp} and \eqref{eq: growth condition estimate for G in local wp}
					for admissible choices of parameters. For the sake of brevity, we treat the previously distinguished cases $s_v \in (-1,0]$ and $s_v \in (0,1]$ at once. 
					We set $m_G=4$ and for the estimate for $g^1$, we have
					\begin{equation*}
						\varphi_2=\beta_2=1-\frac{1+\kappa}{2p},\quad   \rho_2=0,
					\end{equation*}
					so that the right-hand sides of \eqref{Eq1010} and \eqref{Eq10101} match the first summand of \eqref{eq: Lipschitz type estimate for G in local wp} and \eqref{eq: growth condition estimate for G in local wp}, respectively. Since this is the same as in \eqref{eqn43} but with lower $\rho$, the parameter restrictions are met. 
					Because the situation for $g^2$ is more delicate, we first note that when $s_v\in(-1,0]$, the estimate in Lemma \ref{lemma: estimate for g^2 if sv is in -1 0} applies, for which we have the choices of parameters
					\begin{equation*}
						\varphi_3=\beta_3=1-\frac{1+\kappa}{2p},\quad   \rho_3=1,
					\end{equation*}
					which again satisfies \eqref{eq: parameter conditions for local wp} and the (sub)-criticality conditions since they are the same parameters as \eqref{eqn43}. We remark that the above choice of critical parameters is entirely for convenience and could be avoided by sharpening Lemma \ref{lemma: estimate for g^2 if sv is in -1 0}.
					
					Finally, for the estimate of $g^2$ when $s_v\in (0,1]$, we recall that the estimate \eqref{eqn_LL_g2_case2} is of the form 
					\begin{align}\begin{split}\label{eq: estimate for g2 in the local wp proof}
							&\|g^2(v) - g^2(\bar v)\|_{H^{1+s_v,q_v}(\mathbb{T}^d;\ell^2)} \\
							& \lesssim_{(\eta ,g^2)}\Big(1 + \|v\|_{H^{2,q_v}(\mathbb{T}^d)}+\|\bar{v}\|_{H^{2,q_v}(\bT^d)} \Big)  \|v - \bar v\|_{H^{\zeta_\eta,q_v}(\mathbb{T}^d)}
							\\&
							\quad + \left(\|v\|_{H^{1+\frac{d}{2q_v},q_v}(\mathbb{T}^d)}^{1+s_v}+\|\bar{v}\|_{H^{1+\frac{d}{2q_v},q_v}(\bT^d)}^{1+s_v}\right)\|v - \bar v\|_{H^{\zeta_\eta,q_v}(\mathbb{T}^d)},
						\end{split}
					\end{align}
					where $\zeta_\eta = s_v+2-2\frac{1+\kappa}{p}+\eta$, and we choose in the following  
					\begin{equation}\label{eq: epsilon small for local wp}
						\eta\in\left(0,\min\left\{s_v,\frac{2 (1+\kappa)}{p(2+s_v)}\right\}\right).
					\end{equation}
					Since $s_v, \frac{1+\kappa}{p}>0$, we note that the minimum of the two quantities is strictly positive, so that the above interval is non-empty. As opposed to the preceding choices, we aim to consider subcritical parameter sets in the following section to facilitate the proofs. 
					
					For the terms in the first line on the right hand side of \eqref{eq: estimate for g2 in the local wp proof}, we have the choices of parameters 
					\begin{equation*}
						\varphi_4=1-\min\left\{\frac{s_v}{2}, \frac{1+\kappa}{p}-\frac{\eta}{2}\right\},\quad  \beta_4=1-\frac{1+\kappa}{p}+\frac{\eta}{2},\quad   \rho_4=1,
					\end{equation*}
					where the minimum is in place to ensure that  $\varphi_4$ and $\beta_4$  satisfy \eqref{eq: parameter conditions for local wp}.
					For the subcriticality condition \eqref{eq: subcriticality condition in local wp proof}, we observe
					\begin{align*}&
						1\left(1-\min\left\{\frac{s_v}{2}, \frac{1+\kappa}{p}-\frac{\eta}{2}\right\}-1+\frac{1+\kappa}{p}\right)+1-\frac{1+\kappa}{p}+\frac{\eta}{2}\\& \quad = 
						1-\min\left\{\frac{s_v}{2}, \frac{1+\kappa}{p}-\frac{\eta}{2}\right\}+\frac{\eta}{2}< 1,
					\end{align*}
					as in particular 
					$$\eta\in\left(0,\min\left\{s_v,\frac{1+\kappa}{p}\right\}\right),
					$$
					by $s_v>0$.
					
					Finally, to verify that the terms in the final line of \eqref{eq: estimate for g2 in the local wp proof} satisfy the required estimates subcritically, we 
					introduce
					\begin{equation}\label{eq: definition of nu local wp}
						\nu := \frac{1+s_v}{2+s_v} + \varepsilon, 
					\end{equation}
					for
					\begin{equation}\label{eq: condition on epsilon' in proof of local wp}
						\varepsilon \in \left(0 , \min \left\{ \frac{p}{2 (1+\kappa)} \left(\frac{2(1+\kappa)}{p(2+s_v)} -\eta \right) ,
						\frac{s_v p}{4(1+\kappa)} \biggl( 1 - \frac{2(1+\kappa)}{p(2+s_v)} \biggr)
						\right\}\right).
					\end{equation}
					We realize that admissible choices of $\epsilon$ exist, since both of the terms in the minimum are positive. For the first, this follows from the fact that $\eta<\frac{2(1+\kappa)}{p(2+s_v)}$ from condition \eqref{eq: epsilon small for local wp}, and for the second this follows from $s_v>0$ and $\frac{1+\kappa}{p}\leq 1/2$, so that $1 - \frac{2(1+\kappa)}{p(2+s_v)}>1/2$.
					In order to pick the parameters $(\varphi_5, \beta_5, \rho_5)$ for  the final line of \eqref{eq: estimate for g2 in the local wp proof}, we apply the   embeddings 
					\begin{align} \begin{split}\label{eqn843265325}&
							H^{2(1-\nu\frac{1+\kappa}{p})+s_v,q_v}(\bT^d) \hookrightarrow H^{\zeta_\eta,q_v}(\bT^d), \\&  H^{2(1-\nu\frac{1+\kappa}{p})+s_v,q_v}(\bT^d) \hookrightarrow H^{1+\frac{d}{2q_v},q_v}(\bT^d),\end{split}
					\end{align}
					ensured by our choice of $\varepsilon$. Indeed, the former holds if
					\begin{align*}&s_v+2-2\frac{1+\kappa}{p}+\eta = 
						\zeta_\eta \leq 2\left(1-\nu\frac{1+\kappa}{p}\right)+s_v.
					\end{align*}
					Rearranging and using the definition of $\nu$ from \eqref{eq: definition of nu local wp} gives the equivalent 
					\begin{equation*}
						\frac{2(1+\kappa)}{p}\nu \leq \frac{2(1+\kappa)}{p}-\eta\iff \frac{2(1+\kappa)}{p}\epsilon\leq \frac{2(1+\kappa)}{p(2+s_v)}-\eta.
					\end{equation*}
					Multiplying both sides by $\frac{p}{2(1+\kappa)}$, we see that the inequality is guaranteed by the first term in the minimum in the smallness condition of $\epsilon$ in \eqref{eq: condition on epsilon' in proof of local wp}.

					For the second embedding, we require $1+\frac{d}{2q_v} \le 2\left(1-\nu\frac{1+\kappa}{p}\right)+s_v$.
					Using \eqref{Eq_pc2_new}, this condition is equivalent to 
					\[1+\frac{1}{2}\left( s_v + 2 - 2\frac{1+\kappa}{p}\right) \le s_v + 2 - 2\nu\frac{1+\kappa}{p}.
					\]
					Again, rearranging and using the definition of $\nu$ from \eqref{eq: definition of nu local wp}, this is seen to be equivalent to
					\begin{equation*}
						\frac{2(1+\kappa)}{p}\nu \leq \frac{s_v}{2}+\frac{1+\kappa}{p}\iff \frac{2(1+\kappa)}{p}\epsilon\leq \frac{s_v}{2}\left(1-\frac{2(1+\kappa)}{p(2+s_v)}\right).
					\end{equation*}
					Multiplying both sides by $\frac{p}{2(1+\kappa)}$, we find that the inequality is guaranteed by the second term in the minimum in the smallness condition for $\epsilon$, cf.\ \eqref{eq: condition on epsilon' in proof of local wp}.
					
					With the embeddings \eqref{eqn843265325} at hand, it follows that we can choose the parameters
					\begin{equation*}
						\beta_5 = \varphi_5 = 1 - \nu \frac{1+\kappa}{p}, \quad \rho_5=1+s_v.
					\end{equation*}
					The latter satisfy the  parameter condition \eqref{eq: parameter conditions for local wp} since  
					\begin{equation*}
						\varepsilon< \frac{1}{2+s_v},
					\end{equation*}
					and thus $\nu<1$, which is ensured by the first term in the minimum of \eqref{eq: condition on epsilon' in proof of local wp}.
					Furthermore, regarding the subcriticality condition \eqref{eq: subcriticality condition in local wp proof}, we verify
					\begin{align*}\begin{split}\label{eqn_1}
							& (1+s_v)\left(1 -\nu \frac{1+\kappa}{p}-1+\frac{1+\kappa}{p}\right)+1 - \nu \frac{1+\kappa}{p}
							\\& = 1+  (2+s_v)\frac{1+\kappa}{p}(1-\nu)  -\frac{1+\kappa}{p}
							\\&=1-\frac{1+\kappa}{p}\epsilon (2+s_v)<1,
						\end{split}
					\end{align*}
					completing the proof.}
			\end{enumerate}
		\end{proof}

		\section{Proof of long time well-posedness with high probability}\setcounter{equation}{0}\label{sec4}
		
		In this section, we prove Theorem~\ref{thm_global_sol}. 
		Throughout the section, we impose the additional assumption \eqref{assumption_for_global}, namely
		$$
		g^1(0,0,0)=0,
		\qquad
		g^2(0)=0.
		$$
		
		We first prove a short-time persistence estimate. 
		It states that smallness of the initial datum in the critical trace space $X_{1-\frac{1+\kappa}{p},p}$ implies existence up to a fixed time $T_0>0$ with high probability, together with a stopped maximal-regularity estimate.
		Theorem~\ref{thm_global_sol} will then follow by iterating this estimate on finitely many overlapping time intervals, following the strategy of \cite[Theorem~2.11]{Agresti2024StochasticNS}. For notational brevity, we write in the following $U=(u,v)$ and $X_\beta$ and $X_{1-\frac{1+\kappa}{p},p}$ for the complex interpolation space and the trace space \eqref{eq: definition of X beta} and \eqref{t_space} respectively, in line with the more general setting laid out above.

		\begin{lem}\label{lemma:smallData}
			Assume that the parameters $s,s_v,q,q_v,p,\kappa$ satisfy \eqref{Eq_standard} and  \eqref{Eq_pc1_new}--\eqref{Eq_pc5_new}, and let $g^{i}$, $i\in \{1,2\}$, satisfy Assumptions \ref{Ass_noise} and \eqref{assumption_for_global}. 
			Then there exist constants $T_0>0$ and $C_*>0$,  such that, for every $\varepsilon\in(0,1)$, if
			\begin{equation}\label{eqn47}
				U_0\in L^p_{\mathfrak{F}_0}\bigl(\Omega; X_{1-\frac{1+\kappa}{p},p}\bigr)
				\quad\text{and}\quad
				\bE\|U_0\|_{X_{1-\frac{1+\kappa}{p},p}}^p \le C_*\varepsilon,
			\end{equation}
			then the corresponding $L^p_{\kappa}$-maximal solution $(U,\sigma)$ to the stochastic Keller--Segel system \eqref{25.06.19.16.01} provided by Theorem \ref{thm:local_wp} satisfies
			\begin{equation}
				\label{eqn3245}
				\bP(\sigma \geq T_0)\ge 1-\varepsilon.
			\end{equation}
			Moreover, there exists a stopping time $\tau\in(0,\sigma]$ and  constants $K_\theta<\infty$ for $\theta \in [0,1/2)$,  such that \eqref{eqn47} implies also
			\begin{equation}\label{eqn9}
				\bP(\tau\ge T_0)\ge 1-\varepsilon,
			\end{equation}
			and we have the estimate
			\begin{equation}\label{eqn10}
				\bE\left[
				\mathbf{1}_{\{\tau\ge \bar T_0\}}
				\|U\|^p_{H^{\theta,p}((0,\bar T_0),w_\kappa;X_{1-\theta})}
				\right]
				\leq
				K_\theta\,\bE\|U_0\|_{X_{1-\frac{1+\kappa}{p},p}}^p,
			\end{equation}
			for all  $\bar T_0\in(0,T_0]$. In the case $p=2$, the left-hand side of \eqref{eqn10} needs to  be replaced by $\bE[
			\mathbf{1}_{\{\tau\ge \bar T_0\}}
			\|U\|^p_{C([0,\bar T_0];X_{1/2})}
			]$ for $\theta>0$.
		\end{lem}
		\begin{proof}
			{
				We give the proof in the case $p\neq 2$. 
				The case $p=2$ is analogous, with the $H^{\theta,p}$-estimate replaced by the corresponding estimate in $C([0,\bar T_0];X_{1/2})$.
				
				Set
				$$
				\beta:=1-\frac{1+\kappa}{2p}.
				$$
				We first record the nonlinear estimate, which we will use throughout the proof. For this, we crucially rely on the additional assumption \eqref{assumption_for_global}, since it eliminates the additional constants in the growth bounds that we proved. Indeed, Lemma~\ref{lem: bound for F1} and revisiting the proofs of Lemma~\ref{lem: bound for g1}, Lemma~\ref{lemma: estimate for g^2 if sv is in -1 0}, and Lemma~\ref{lemma: estimate for g^2 if sv is in 0 1} 
				yields the following
				pointwise bounds for
				$$
				N(U):=\|F(U)\|_{X_0}+\|G(U)\|_{\gamma(H,X_{1/2})}.
				$$
				If $s_v\in(-1,0]$, then
				\begin{equation*}\label{eq: N estimate sv nonpositive}
					N(U)
					\leq
					C\left(
					\|U\|_{X_\beta}
					+
					\|U\|_{X_\beta}^2
					\right).
				\end{equation*}
				If $s_v\in(0,1]$, then, with the parameter choices 
				$$
				\beta_4
				=
				1-\frac{1+\kappa}{p}+\frac{\eta}{2},
				\qquad
				\varphi_4
				=
				1-
				\min\left\{
				\frac{s_v}{2},
				\frac{1+\kappa}{p}-\frac{\eta}{2}
				\right\},
				$$
				and
				$$
				\beta_5
				=
				1-\nu\frac{1+\kappa}{p}.
				$$
				from the proof of Theorem~\ref{thm:local_wp}, we have
				\begin{equation}\label{eq: N estimate sv positive}
					N(U)
					\leq
					C\left(
					\|U\|_{X_\beta}
					+
					\|U\|_{X_\beta}^2
					+
					\|U\|_{X_{\beta_4}}\|U\|_{X_{\varphi_4}}
					+
					\|U\|_{X_{\beta_5}}^{2+s_v}
					\right).
				\end{equation}
				Here $\eta>0$ is chosen in accordance with \eqref{eq: epsilon small for local wp}, and $\nu$ is as in \eqref{eq: definition of nu local wp}. 
				Decreasing $\eta>0$ if necessary, we additionally assume that
				\begin{equation}\label{eq: eta small for short time lemma}
					\eta
					<
					\min\left\{
					\frac{s_v}{2},
					\frac{2}{3}\frac{1+\kappa}{p}
					\right\}.
				\end{equation}
				When $s_v\in(-1,0]$, the terms involving $\beta_4,\varphi_4,\beta_5$ are simply absent. 
				Thus, it suffices to write the argument in the form which covers the case $s_v\in(0,1]$.
				
				We now introduce the auxiliary space used in the stopping-time argument. 
				For $t<\sigma$, define
				\begin{align*}
					\label{eq: definition of mathscr X}
					\mathscr X(t)
					:={}&
					L^{2p}\bigl((0,t),w_\kappa;X_\beta\bigr)
					\cap
					L^{(\rho_{\beta_4}+1)p}
					\bigl((0,t),w_\kappa;X_{\beta_4}\bigr)
					\nonumber\\
					&\quad \cap
					L^{(\rho_{\varphi_4}+1)p}
					\bigl((0,t),w_\kappa;X_{\varphi_4}\bigr)
					\cap
					L^{(\rho_{\beta_5}+1)p}
					\bigl((0,t),w_\kappa;X_{\beta_5}\bigr),
				\end{align*}
				endowed with the natural sum norm. 
				The exponents are chosen so that
				\begin{equation}\label{eq: rho definition general}
					(\rho_\vartheta+1)
					\left(
					\vartheta-\left(1-\frac{1+\kappa}{p}\right)
					\right)
					=
					\frac{1+\kappa}{p},
					\qquad
					\vartheta\in\{\beta_4,\varphi_4,\beta_5\}.
				\end{equation}
				Equivalently,
				\begin{equation}\label{eq: rho explicit}
					\rho_{\beta_4}+1
					=
					\frac{2(1+\kappa)}{p\eta},
					\quad
					\rho_{\varphi_4}+1
					=
					\frac{1+\kappa}{
						1+\kappa
						-
						p\min\left\{
						\frac{s_v}{2},
						\frac{1+\kappa}{p}-\frac{\eta}{2}
						\right\}},
					\quad
					\rho_{\beta_5}+1
					=
					\frac{1}{1-\nu}.
				\end{equation}
				The definition \eqref{eq: rho definition general} is chosen so that the usual interpolation argument applies. 
				More precisely, for
				$\vartheta\in\{\beta_4,\varphi_4,\beta_5\}$,
				\begin{equation}\label{eq: interpolation into mathscr X}
					\|U\|_{L^{(\rho_\vartheta+1)p}((0,t),w_\kappa;X_\vartheta)}
					\lesssim
					\|U\|_{C([0,t];X_{1-\frac{1+\kappa}{p},p})}^{1-\theta_\vartheta}
					\|U\|_{L^p((0,t),w_\kappa;X_1)}^{\theta_\vartheta},
				\end{equation}
				where
				$$
				\theta_\vartheta
				=
				\frac{
					\vartheta-\left(1-\frac{1+\kappa}{p}\right)
				}{
					\frac{1+\kappa}{p}
				}
				\in(0,1).
				$$
				In particular, by the local regularity of $U$, we have
				$$
				U\in \mathscr X(t)
				\qquad
				\text{for every }t<\sigma,\quad {\text{almost surely.}}
				$$
				
				We next derive the time-integrated estimate for $N(U)$. 
				By Young's inequality,
				\begin{equation*}\label{eq: young beta4 phi4}
					\|U\|_{X_{\beta_4}}\|U\|_{X_{\varphi_4}}
					\lesssim
					\|U\|_{X_{\beta_4}}^{\frac{1+\kappa}{p\eta}}
					+
					\|U\|_{X_{\varphi_4}}^{
						\frac{1+\kappa}{1+\kappa-p\eta}
					}.
				\end{equation*}
				Inserting this into \eqref{eq: N estimate sv positive}, we obtain
				\begin{equation*}\label{eq: N estimate after Young}
					N(U)
					\leq
					C\left(
					\|U\|_{X_\beta}
					+
					\|U\|_{X_\beta}^2
					+
					\|U\|_{X_{\beta_4}}^{\frac{1+\kappa}{p\eta}}
					+
					\|U\|_{X_{\varphi_4}}^{
						\frac{1+\kappa}{1+\kappa-p\eta}
					}
					+
					\|U\|_{X_{\beta_5}}^{2+s_v}
					\right).
				\end{equation*}
				By  \eqref{eq: eta small for short time lemma} and \eqref{eq: rho explicit}, the powers appearing in the subcritical terms are strictly smaller than the corresponding time-integrability exponents in the definition of $\mathscr X(t)$.
				More precisely,
				$$
				\frac{1+\kappa}{p\eta}
				<
				\rho_{\beta_4}+1,
				\qquad
				\frac{1+\kappa}{1+\kappa-p\eta}
				<
				\rho_{\varphi_4}+1,
				\qquad
				2+s_v
				<
				\rho_{\beta_5}+1.
				$$
				Thus the terms
				$$
				\|U\|_{X_{\beta_4}}^{\frac{1+\kappa}{p\eta}},
				\qquad
				\|U\|_{X_{\varphi_4}}^{
					\frac{1+\kappa}{1+\kappa-p\eta}
				},
				\qquad
				\|U\|_{X_{\beta_5}}^{2+s_v}
				$$
				are strictly subcritical with respect to the corresponding components of
				$\mathscr X(t)$.
				The linear term $\|U\|_{X_\beta}$ is also subcritical with respect to $L^{2p}((0,t),w_\kappa;X_\beta)$, whereas the quadratic term $\|U\|_{X_\beta}^2$ is critical. 
				Hence, H\"older's inequality gives that there
				exist {constants \(C<\infty\) }and \(\iota>0\) such that, for all \(t\leq 1\),
				\begin{equation}
					\label{eq: integrated N estimate}
					\begin{aligned}&
						\|N(U)\|_{L^p((0,t),w_\kappa)}^p \\&\quad 
						\leq{}
						C t^\iota
						\left(
						\|U\|_{\mathscr X(t)}^p
						+
						\|U\|_{\mathscr X(t)}^{\frac{1+\kappa}{\eta}}
						+
						\|U\|_{\mathscr X(t)}^{
							p\frac{1+\kappa}{1+\kappa-p\eta}
						}
						+
						\|U\|_{\mathscr X(t)}^{(2+s_v)p}
						\right)
						+
						C\|U\|_{\mathscr X(t)}^{2p}.
					\end{aligned}
				\end{equation}
				We remark that in the case $s_v\in(-1,0]$, the same estimate holds with the subcritical terms corresponding to $\beta_4,\varphi_4,\beta_5$ omitted.
				
				Let $T_0\leq 1$ and $r_0>0$ be chosen later. 
				Define
				\begin{equation*}\label{eq: definition good event}
					\mathcal O_{T_0,r_0}
					:=
					\left\{
					\omega\in\Omega:
					\|U(\omega)\|_{\mathscr X(\sigma(\omega)\wedge T_0)}
					\leq r_0
					\right\}.
				\end{equation*}
				By \eqref{eq: integrated N estimate}, on the event $\mathcal O_{T_0,r_0}$ we have
				\begin{equation}\label{eq: N bounded on good event}
					\begin{aligned}
						\|N(U)\|_{L^p((0,\sigma\wedge T_0),w_\kappa)}^p
						&\leq
						C T_0^\iota
						\left(
						r_0^p
						+
						r_0^{\frac{1+\kappa}{\eta}}
						+
						r_0^{p\frac{1+\kappa}{1+\kappa-p\eta}}
						+
						r_0^{(2+s_v)p}
						\right)
						+
						C r_0^{2p}
						\\
						&=: C_1(T_0,r_0).
					\end{aligned}
				\end{equation}
				
				{By means of the blow-up criteria from \cite{AV_SQEE_pt2} one can argue that}
				\begin{equation}\label{eq: no blowup on good event}
					\bP\bigl(\{\sigma<T_0\}\cap\mathcal O_{T_0,r_0}\bigr)=0.
				\end{equation}
				Indeed, by \eqref{eq: N bounded on good event}, on the event $\mathcal O_{T_0,r_0}$ we have
				$$
				\|N(U)\|_{L^p((0,\sigma\wedge T_0),w_\kappa)}^p
				\leq
				C_1(T_0,r_0)<\infty.
				$$
				Hence, on \(\{\sigma<T_0\}\cap\mathcal O_{T_0,r_0}\),
				$$
				\|N(U)\|_{L^p((0,\sigma),w_\kappa)}<\infty.
				$$
				Therefore,
				$$
				\{\sigma<T_0\}\cap\mathcal O_{T_0,r_0}
				\subset
				\left\{
				\sigma<T_0,\,
				\|N(U)\|_{L^p((0,\sigma),w_\kappa)}<\infty
				\right\},
				$$
				{and 
					by the semilinear blow-up criterion \cite[Theorem~4.10(1)]{AV_SQEE_pt2}, the latter event has probability zero.
					This proves the claimed \eqref{eq: no blowup on good event}.}
				
				We now estimate the probability of the complement of the event $\mathcal O_{T_0,r_0}$. 
				For each $n\geq 1$, define the stopping times
				\begin{equation*}
					\label{eq: definition sigma n mu mu n}
					\begin{aligned}
						\sigma_n
						&:=
						\inf\left\{
						t\in[0,\sigma):
						\|U\|_{L^p((0,t),w_\kappa;X_1)}^p
						\geq n
						\right\}\wedge \sigma,
						\\
						\mu
						&:=
						\inf\left\{
						t\in[0,\sigma):
						\|U\|_{\mathscr X(t)}
						\geq r_0
						\right\}\wedge T_0\wedge\sigma,
						\\
						\mu_n
						&:=
						\mu\wedge\sigma_n.
					\end{aligned}
				\end{equation*}
				Here and below, we use the convention that the infimum of the empty set is $\infty$. 
				Since $(U,\sigma)$ is a local solution, the {$L^p((0,t), w_\kappa; X_1)$-norm  of $U$ is finite} for every $t<\sigma$. 
				Hence
				$$
				\sigma_n\uparrow\sigma
				\qquad\text{almost surely as }n\to\infty,
				$$
				{and as a consequence}
				$$
				\mu_n=\mu\wedge\sigma_n\uparrow\mu
				\qquad\text{almost surely as }n\to\infty.
				$$
				Moreover, by the definition of $\mu$ and the continuity of $t\mapsto \|U\|_{\mathscr X(t)}$, we have
				\begin{equation}\label{eq: mu bound by r0}
					\|U\|_{\mathscr X(\mu_n)}
					\leq
					\|U\|_{\mathscr X(\mu)}
					\leq r_0
					\qquad\text{almost surely}.
				\end{equation}
				
				Applying \cite[Lemma~5.3]{AV_SQEE_pt2} on the stochastic interval $[0,\mu_n]$, we obtain
				$$
				\begin{aligned}
					&\bE\|U\|_{L^p((0,\mu_n),w_\kappa;X_1)}^p
					+
					\bE\|U\|_{C([0,\mu_n];X_{1-\frac{1+\kappa}{p},p})}^p
					\\
					&\qquad\leq
					C\left(
					\bE\|U_0\|_{X_{1-\frac{1+\kappa}{p},p}}^p
					+
					\bE\|N(U)\|_{L^p((0,\mu_n),w_\kappa)}^p
					\right).
				\end{aligned}
				$$
				Together with the interpolation estimate \eqref{eq: interpolation into mathscr X}, this implies
				$$
				\bE\|U\|_{\mathscr X(\mu_n)}^p
				\leq
				C\left(
				\bE\|U_0\|_{X_{1-\frac{1+\kappa}{p},p}}^p
				+
				\bE\|N(U)\|_{L^p((0,\mu_n),w_\kappa)}^p
				\right).
				$$
				Since $\mu_n\leq T_0\leq 1$, the estimate
				\eqref{eq: integrated N estimate} gives
				\begin{equation}
					\label{eq: X mu n bootstrap before absorption}
					\begin{aligned}
						\bE\|U\|_{\mathscr X(\mu_n)}^p
						\leq{}&
						C\bE\|U_0\|_{X_{1-\frac{1+\kappa}{p},p}}^p\\
						&+
						C T_0^\iota
						\bE\Bigl[
						\|U\|_{\mathscr X(\mu_n)}^p
						+
						\|U\|_{\mathscr X(\mu_n)}^{\frac{1+\kappa}{\eta}}
						+
						\|U\|_{\mathscr X(\mu_n)}^{
							p\frac{1+\kappa}{1+\kappa-p\eta}
						}
						+
						\|U\|_{\mathscr X(\mu_n)}^{(2+s_v)p}
						\Bigr]\\
						&+
						C\bE\|U\|_{\mathscr X(\mu_n)}^{2p}.
					\end{aligned}
				\end{equation}
				Set
				$$
				Z_n:=\|U\|_{\mathscr X(\mu_n)}^p.
				$$
				Then \eqref{eq: mu bound by r0} implies
				$$
				0\leq Z_n\leq r_0^p
				\qquad\text{almost surely}.
				$$
				In terms of $Z_n$, \eqref{eq: X mu n bootstrap before absorption} becomes
				\begin{equation}
					\label{eq: Zn bootstrap before absorption}
					\begin{aligned}
						\bE Z_n
						\leq{}&
						C\bE\|U_0\|_{X_{1-\frac{1+\kappa}{p},p}}^p+
						C T_0^\iota
						\bE\left[
						Z_n
						+
						Z_n^{\frac{1+\kappa}{p\eta}}
						+
						Z_n^{\frac{1+\kappa}{1+\kappa-p\eta}}
						+
						Z_n^{2+s_v}
						\right]
						+
						C\bE Z_n^2.
					\end{aligned}
				\end{equation}
				Choose $R=2C$, where $C$ is the constant in \eqref{eq: Zn bootstrap before absorption}, and then choose $r_0,\bar{r}_0>0$ by
				\begin{equation}\label{eq: choice r0}
					r_0^p:=\frac{1}{2R},\quad \bar{r}_0^p:=\frac{3}{4R}.
				\end{equation}
				We choose $T_0\leq 1$ sufficiently small so that
				\begin{equation}\label{eq: choice T0 absorption}
					C T_0^\iota
					\left(
					1
					+
					\bar{r}_0^{\frac{1+\kappa}{\eta}-p}
					+
					\bar{r}_0^{
						\frac{p(1+\kappa)}{1+\kappa-p\eta}-p
					}
					+
					\bar{r}_0^{(1+s_v)p}
					\right)
					\leq
					\frac{1}{2}.
				\end{equation}
				Using $Z_n\leq r_0^p\leq \bar{r}_0^p$, the choice \eqref{eq: choice T0 absorption} allows us to absorb the $T_0^\iota$-terms in \eqref{eq: Zn bootstrap before absorption}.
				Thus
				$$
				\bE Z_n
				\leq
				2C\bE\|U_0\|_{X_{1-\frac{1+\kappa}{p},p}}^p
				+
				2C\bE Z_n^2.
				$$
				Since $R=2C$, this is equivalent to
				\begin{equation}\label{eq: psi R Zn estimate}
					\bE\bigl[\psi_R(Z_n)\bigr]
					\leq
					\bE\|U_0\|_{X_{1-\frac{1+\kappa}{p},p}}^p,
					\qquad
					\psi_R(x):=\frac{x}{R}-x^2.
				\end{equation}
				
				{To take the limit $n\to\infty$ in the above estimate, we observe that by } \eqref{eq: choice r0} we have $Z_n\leq 1/(2R)$ almost surely. 
				In particular, $(Z_n)_{n\geq1}$ is nondecreasing and takes values in $[0,1/(2R)]$. 
				Since $\psi_R'\geq0$ on $[0,1/(2R)]$, the function $\psi_R$ is nondecreasing and nonnegative on this interval. 
				Hence, by the monotone convergence theorem,
				\begin{equation}\label{eq: psi R mu estimate}
					\bE\left[
					\psi_R\left(\|U\|_{\mathscr X(\mu)}^p\right)
					\right]
					\leq
					\bE\|U_0\|_{X_{1-\frac{1+\kappa}{p},p}}^p.
				\end{equation}
				
				The preceding bound may then be used to show that the probability of $\mathcal O_{T_0,r_0}$  is large for small initial data as follows: On its complement 
				$\Omega\setminus\mathcal O_{T_0,r_0}$, we have $\mu<\sigma\wedge T_0$.  By the continuity of $t\mapsto \|U\|_{\mathscr X(t)}$ and the definition of $\mu$, this entails
				$$
				\|U\|_{\mathscr X(\mu)}^p
				=
				r_0^p
				=
				\frac{1}{2R}.
				$$
				Therefore,
				$$
				\psi_R\left(\|U\|_{\mathscr X(\mu)}^p\right)
				=
				\psi_R\left(\frac{1}{2R}\right)
				=
				\frac{1}{4R^2}
				\qquad\text{on }\Omega\setminus\mathcal O_{T_0,r_0},
				$$
				{and 
					combining this with \eqref{eq: psi R mu estimate}, we obtain
					\begin{equation}\label{eq: good event complement probability}
						\frac{1}{4R^2}\,
						\bP(\Omega\setminus\mathcal O_{T_0,r_0})
						\leq
						\bE\|U_0\|_{X_{1-\frac{1+\kappa}{p},p}}^p.
					\end{equation}
					It remains to set
					\begin{equation*}\label{eq: definition C star}
						C_*:=\frac{1}{4R^2}.
					\end{equation*}
					Then, if 
					$$
					\bE\|U_0\|_{X_{1-\frac{1+\kappa}{p},p}}^p
					\leq
					C_*\varepsilon,
					$$
					\eqref{eq: good event complement probability} implies
					\begin{equation}\label{eq: good event high probability}
						\bP(\mathcal O_{T_0,r_0})\geq 1-\varepsilon.
					\end{equation}
					Together with \eqref{eq: no blowup on good event}, this proves  the first part of the assertion \eqref{eqn3245}.}
				
				We now prove  {the corresponding  a priori} estimate, for which we recall $\bar{r}_0$ from \eqref{eq: choice r0} and define
				\begin{equation*}\label{eq: definition tau for stopped estimate}
					\tau
					:=
					\inf\left\{
					t\in[0,\sigma):
					\|U\|_{\mathscr X(t)}
					\geq
					\bar{r}_0
					\right\}
					\wedge\sigma.
				\end{equation*}
				Since
				$$
				r_0^p=\frac{1}{2R}
				<
				\frac{3}{4R}= \bar{r}_0^p,
				$$
				we have
				$$
				\tau\wedge T_0
				=
				\sigma\wedge T_0
				\qquad
				\text{on }\mathcal O_{T_0,r_0}.
				$$
				Consequently, by \eqref{eq: no blowup on good event} and
				\eqref{eq: good event high probability},
				$$
				\bP(\tau\geq T_0)
				\geq
				\bP(\{\sigma\geq T_0\}\cap\mathcal O_{T_0,r_0})
				=
				\bP(\mathcal O_{T_0,r_0})
				\geq
				1-\varepsilon,
				$$
				proving that also $\tau$ obeys \eqref{eqn9}.
				
				{It remains to show the bound \eqref{eqn10}. }Fix $\bar T_0\in(0,T_0]$ and set
				\begin{equation*}\label{eq: definition Z stopped}
					Z
					:=
					\|U\|_{\mathscr X(\tau\wedge\bar T_0)}^p.
				\end{equation*}
				Repeating the argument leading to \eqref{eq: psi R Zn estimate}, now on the stochastic interval $[0,\tau\wedge\bar T_0]$, we obtain
				\begin{equation}\label{eq: psi R Z stopped}
					\bE[\psi_R(Z)]
					\leq
					\bE\|U_0\|_{X_{1-\frac{1+\kappa}{p},p}}^p.
				\end{equation}
				By the definition of $\tau$,
				\begin{equation}\label{eq: Z stopped bound}
					0\leq Z\leq \bar{r}_0^p=\frac{3}{4R}
					\qquad\text{almost surely}.
				\end{equation}
				For $x\in[0,3/(4R)]$, we have
				$$
				\psi_R(x)=\frac{x}{R}-x^2
				\geq
				\frac{x}{4R}.
				$$
				Therefore, \eqref{eq: psi R Z stopped} gives
				\begin{equation}\label{eq: EZ stopped bound}
					\bE Z
					\leq
					4R\,\bE\|U_0\|_{X_{1-\frac{1+\kappa}{p},p}}^p,
				\end{equation}
				which in combination with  \eqref{eq: Z stopped bound} yields
				\begin{equation}\label{eq: EZ square stopped bound}
					\bE Z^2
					\leq
					\frac{3}{4R}\bE Z
					\leq
					3\,\bE\|U_0\|_{X_{1-\frac{1+\kappa}{p},p}}^p.
				\end{equation}
				
				{To deduce \eqref{eqn10}, we observe that  on} the event $\{\tau\geq \bar T_0\}$, we have
				$\tau\wedge\bar T_0=\bar T_0$. {We remark that it is precisely at this point, where we profit from defining $\tau$ in terms of the larger $\bar{r}_0$ instead of $r_0$. Consequently}, using
				\eqref{eq: integrated N estimate} and the choice of $T_0$ in
				\eqref{eq: choice T0 absorption}, we obtain
				\begin{equation*}\label{eq: stopped N estimate by Z}
					\mathbf 1_{\{\tau\geq \bar T_0\}}
					\|N(U)\|_{L^p((0,\bar T_0),w_\kappa)}^p
					\leq
					\frac12 Z + C Z^2.
				\end{equation*}
				Taking expectations and using \eqref{eq: EZ stopped bound} and
				\eqref{eq: EZ square stopped bound}, we get
				\begin{equation}\label{eq: stopped N expectation bound}
					\bE\left[
					\mathbf 1_{\{\tau\geq \bar T_0\}}
					\|N(U)\|_{L^p((0,\bar T_0),w_\kappa)}^p
					\right]
					\leq
					C\bE\|U_0\|_{X_{1-\frac{1+\kappa}{p},p}}^p.
				\end{equation}
				
				Finally, define the stopped nonlinearities
				$$
				F_\tau
				:=
				\mathbf 1_{[0,\tau\wedge\bar T_0)}F(U),
				\qquad
				G_\tau
				:=
				\mathbf 1_{[0,\tau\wedge\bar T_0)}G(U).
				$$
				By \eqref{eq: stopped N expectation bound},
				$$
				F_\tau
				\in
				L^p\bigl(\Omega;L^p((0,\bar T_0),w_\kappa;X_0)\bigr),
				\qquad
				G_\tau
				\in
				L^p\bigl(\Omega;L^p((0,\bar T_0),w_\kappa;\gamma(H,X_{1/2}))\bigr).
				$$
				Since $A\in\mathcal{SMR}_{p,\kappa}^{\bullet}$ by
				Theorem~\ref{thm: A B in SMR}, stochastic maximal regularity and the
				localization property \cite[Propositions~3.10 and 3.12]{AV_SQEE_pt2} yield,
				for each $\theta\in[0,1/2)$,
				$$
				\bE\left[
				\mathbf 1_{\{\tau\geq \bar T_0\}}
				\|U\|_{H^{\theta,p}((0,\bar T_0),w_\kappa;X_{1-\theta})}^p
				\right]
				\leq
				K_\theta
				\bE\|U_0\|_{X_{1-\frac{1+\kappa}{p},p}}^p.
				$$
				This proves \eqref{eqn10} and completes the proof of the lemma.}
		\end{proof}

		\begin{proof}[Proof of Theorem \ref{thm_global_sol}]As for the preceding lemma we give the proof for $p\ne 2$, while the case that $p=2$ may be treated analogously. 
			{Fix an arbitrary desired existence time {$T<\infty$}.
				We prove the theorem by iterating Lemma \ref{lemma:smallData} on successive time intervals of length $T_0>0$ specified in the lemma. 
				Define $N_0:=2\Bigl\lceil \frac{T}{T_0}\Bigr\rceil$, so that after at most $N_0$ iterations we can cover the whole interval $[0,T]$.

				Next, fix $\theta_0\in\Bigl(\frac{1+\kappa}{p},\frac12\Bigr)$.
				By \cite[Proposition 2.1]{AV25}, there exists a constant {$c_0>0$} such that
				\begin{equation}\label{eq:traceEmbeddingEstimate}
					\sup_{t\in[0,T_0]}
					\|v(t)\|_{X_{1-\frac{1+\kappa}{p},p}}
					\leq
					c_0\Bigl(
					\|v\|_{H^{\theta_0,p}((0,T_0),w_\kappa;X_{1-\theta_0})}
					+
					\|v\|_{L^p((0,T_0),w_\kappa;X_1)}
					\Bigr)
				\end{equation}
				for every $v\in H^{\theta_0,p}(0,T_0,w_\kappa;X_{1-\theta_0})
				\cap
				L^p(0,T_0,w_\kappa;X_1)$.
				We then define
				$$
				C_1:=\bigl((K_{\theta_0}+K_0)c_0\bigr)\vee 1,
				$$
				where $K_{\theta_0}$ and $K_0$ are the constants {$K_\theta$} appearing in Lemma \ref{lemma:smallData} {for the choices $\theta=\theta_0$ and $0$ respectively}. 
				To distribute the admissible error over the whole iteration scheme, we set
				$$
				\varepsilon_0:=\frac{\varepsilon}{2^{N_0}}.
				$$
				We then choose the smallness threshold for the initial datum as
				$$
				C_T
				:=
				\frac{C_*}{2^{N_0+1}C_1^{N_0}},
				$$
				where $C_*$ is the constant from Lemma \ref{lemma:smallData}.

				Let
				$$
				0=t_0<t_1<\cdots<t_{N_0}=T
				$$
				be a partition of $[0,T]$ with mesh size at most $T_0/2$.
				We shall construct a sequence of stopping times inductively
				$\{\tau_n\}_{n=1}^{N_0}$, with $\tau_n\leq \sigma$, such that, for every
				$n\in\{1,\dots,N_0\}$ and every $\theta\in[0,1/2)$,
				\begin{align}
					\bP(\tau_n\geq t_n)
					&>1-2^{\,n}\varepsilon_0
					=1-2^{\,n-N_0}\varepsilon,\label{eq:inductiveProb}\\
					\bE\Bigl[
					\mathbf 1_{\{\tau_n\geq t_n\}}
					\|U(t_n)\|_{X_{1-\frac{1+\kappa}{p},p}}^p
					\Bigr]
					&\leq C_1^{n-N_0}C_*\varepsilon_0,\label{eq:inductiveTrace}\\
					\bE\Bigl[
					\mathbf 1_{\{\tau_n\geq t_n\}}
					\|U\|_{H^{\theta,p}((0,t_n),w_\kappa;X_{1-\theta})}^p
					\Bigr]
					&\lesssim_\theta
					\bE\|U_0\|_{X_{1-\frac{1+\kappa}{p},p}}^p.\label{eq:inductiveReg}
				\end{align}
				For the initial step, we first observe that, by the choice of $C_T$,
				$$
				\bE\|U_0\|_{X_{1-\frac{1+\kappa}{p},p}}^p
				\leq C_T\varepsilon
				=
				\frac{1}{2} C_1^{-N_0}C_*\varepsilon_0
				\leq C_*\varepsilon_0,
				$$
				{where we used in the last estimate that $C_1\ge 1$.}
				Thus, Lemma \ref{lemma:smallData} applies with $\varepsilon_0$ in place of
				$\varepsilon$. Let $\tau$ be the corresponding stopping time and set
				$\tau_1=\tau_2:=\tau$.
				Since $t_1,t_2\leq T_0$, Lemma \ref{lemma:smallData} gives
				$$
				\bP(\tau_n\geq t_n)>1-\varepsilon_0
				\geq 1-2^n\varepsilon_0,
				\qquad n=1,2.
				$$
				Moreover, by the trace embedding \eqref{eq:traceEmbeddingEstimate} and the
				estimates from Lemma \ref{lemma:smallData},
				$$
				\bE\Bigl[
				\mathbf 1_{\{\tau_n\geq t_n\}}
				\|U(t_n)\|_{X_{1-\frac{1+\kappa}{p},p}}^p
				\Bigr]
				\leq
				C_1\bE\|U_0\|_{X_{1-\frac{1+\kappa}{p},p}}^p
				\leq
				\frac12 C_1^{1-N_0}C_*\varepsilon_0
				\leq
				C_1^{n-N_0}C_*\varepsilon_0,
				\qquad n=1,2.
				$$
				Finally, the regularity estimate \eqref{eq:inductiveReg} for $n=1,2$ follows
				directly from Lemma \ref{lemma:smallData}.
				
				We now assume that $\tau_n$ has been constructed for some $n\in\{2,\dots,N_0-1\}$, and we construct $\tau_{n+1}$. 
				The idea is to reapply Lemma \ref{lemma:smallData} starting from time $t_{n-1}$ and $t_n$ in order to extend the solution up to time $t_{n+1}$ with suitable estimates. 
				To this end, define
				$$
				V_{n-1}:=\mathbf 1_{\{\tau_{n-1}\geq t_{n-1}\}}\,U(t_{n-1}).
				$$
				Then, {in light of \eqref{eq:inductiveTrace},}
				$$
				V_{n-1}\in L^p_{\mathcal F_{t_{n-1}}}\bigl(\Omega;X_{1-\frac{1+\kappa}{p},p}\bigr),
				$$
				{with the bound}
				\begin{align*}
					\bE\|V_{n-1}\|_{X_{1-\frac{1+\kappa}{p},p}}^p
					=
					\bE\Bigl[
					\mathbf 1_{\{\tau_{n-1}\geq t_{n-1}\}}
					\|U(t_{n-1})\|_{X_{1-\frac{1+\kappa}{p},p}}^p
					\Bigr]\leq
					C_1^{n-1-N_0}C_*\varepsilon_0
					\leq
					C_*\varepsilon_0,
				\end{align*}
				since $C_1\geq1$ and $n-1\leq N_0$.
				Hence $V_{n-1}$ is admissible as an initial condition for
				Lemma \ref{lemma:smallData} with $\varepsilon_0$ in place of $\varepsilon$ {and the initial time replaced by $t_{n-1}$}.

				By the time-shifted version of Theorem \ref{thm:local_wp}, there exists a $L^p_{\kappa}$-maximal solution $U_{n-1}$ defined on $[t_{n-1},\sigma_{n-1})$, with $\sigma_{n-1}>t_{n-1}$ the time of local existence and with initial condition
				$$
				U_{n-1}(t_{n-1})=V_{n-1}.
				$$
				Let
				$$
				\lambda_{n-1}\in (t_{n-1},\sigma_{n-1}]
				$$
				be the stopping time given by the corresponding time-shifted version of Lemma \ref{lemma:smallData}. Since each interval of the partition has length less or equal than $T_0/2$, we have
				$$
				t_{n+1}-t_{n-1}\leq T_0.
				$$
				Therefore,
				$$
				\bP(\lambda_{n-1}\geq t_{n+1})
				\geq
				\bP(\lambda_{n-1}\geq t_{n-1}+T_0)
				>
				1-\varepsilon_0
				=
				1-2^{-N_0}\varepsilon.
				$$
				Moreover, applying the trace embedding \eqref{eq:traceEmbeddingEstimate} on the interval $[t_{n-1},t_{n+1}]$, and then using Lemma \ref{lemma:smallData}, we obtain
				\begin{equation}\label{eq:boundTraceLambda}
					\begin{aligned}
						&\bE\Bigl[
						\mathbf 1_{\{\lambda_{n-1}\geq t_{n+1}\}}
						\sup_{t\in[t_{n-1},t_{n+1}]}
						\|U_{n-1}(t)\|_{X_{1-\frac{1+\kappa}{p},p}}^p
						\Bigr]\\&
						\leq
						c_0\,
						\bE\Bigl[
						\mathbf 1_{\{\lambda_{n-1}\geq t_{n+1}\}}
						\max_{\theta\in\{0,\theta_0\}}
						\|U_{n-1}\|_{H^{\theta,p}((t_{n-1},t_{n+1}),w_\kappa^{t_{n-1}};X_{1-\theta})}^p
						\Bigr]\\&
						\leq
						C_1\,
						\bE\Bigl[
						\mathbf 1_{\{\tau_{n-1}\geq t_{n-1}\}}
						\|U(t_{n-1})\|_{X_{1-\frac{1+\kappa}{p},p}}^p
						\Bigr]\\&
						\leq
						C_1^{\,n-N_0}C_*\varepsilon_0,
					\end{aligned}
				\end{equation}
				where in the last inequality we used the induction hypothesis
				\eqref{eq:inductiveTrace} with $n-1$ in place of $n$.
				Similarly, define
				$$
				V_n:=\mathbf 1_{\{\tau_n\geq t_n\}}\,U(t_n),
				$$
				and let $U_n$ be the $L^p_{\kappa}$-maximal  solution to \eqref{25.06.19.16.54} on $[t_n,\sigma_n)$ with initial condition
				$$
				U_n(t_n)=V_n.
				$$
				Let $\lambda_n\in(t_n,\sigma_n]$ be the stopping time furnished by the time-shifted version of Lemma \ref{lemma:smallData}. Since
				$$
				t_{n+1}-t_n<T_0,
				$$
				Lemma \ref{lemma:smallData} yields
				$$
				\bP(\lambda_n\geq t_{n+1})>1-\varepsilon_0=1-2^{-N_0}\varepsilon.
				$$
				
				We now define
				\begin{equation}\label{eq: definition of tau n+1}
					\tau_{n+1}=A_{n-1}\wedge A_n,    
				\end{equation}
				where
				$$
				A_n
				=
				\mathbf 1_{\{\tau_n\geq t_n\}}\lambda_n
				+
				{\tau}_n\mathbf 1_{\{\tau_n< t_n\}}, \quad n \in\mathbb{N}.
				$$
				Then $\tau_{n+1}$ is a stopping time satisfying $\tau_{n+1}\leq \sigma$, and by \eqref{eq: definition of tau n+1},
				$$
				\{\tau_{n+1}\geq t_{n+1}\}
				=
				\{A_{n-1}\geq t_{n+1}\}\cap \{A_n\geq t_{n+1}\}.
				$$
				Moreover, since $t_{n-1},t_n<t_{n+1}$,
				$$
				\{A_{n-1}\geq t_{n+1}\}
				=
				\{\tau_{n-1}\geq t_{n-1}\}\cap \{\lambda_{n-1}\geq t_{n+1}\},
				$$
				and
				$$
				\{A_n\geq t_{n+1}\}
				=
				\{\tau_n\geq t_n\}\cap \{\lambda_n\geq t_{n+1}\}.
				$$
				Therefore,
				\begin{equation}
					\label{26.04.16.16.01}
					\{\tau_{n+1}\geq t_{n+1}\}
					=
					\{\tau_{n-1}\geq t_{n-1}\}
					\cap
					\{\tau_n\geq t_n\}
					\cap
					\{\lambda_{n-1}\geq t_{n+1}\}
					\cap
					\{\lambda_n\geq t_{n+1}\}.
				\end{equation}
				Next, we claim that on the event $\{\tau_{n+1}\geq t_{n+1}\}$ one has 
				$$\sigma = \sigma_{n-1}$$
				and
				$$ U(t)=U_{n-1}(t), \qquad t\in [t_{n-1},t_{n+1}].$$
				Indeed, by \eqref{26.04.16.16.01}, the event $\{\tau_{n+1}\geq t_{n+1}\}$ is contained in the event $$ \{\tau_{n-1}\geq t_{n-1}\}\cap \{\lambda_{n-1}\geq t_{n+1}\}. $$
				Hence, on $\{\tau_{n+1}\geq t_{n+1}\}$, $$
				V_{n-1}
				=
				\mathbf 1_{\{\tau_{n-1}\geq t_{n-1}\}}U(t_{n-1})
				=
				U(t_{n-1}),
				$$
				and both $U$ and $U_{n-1}$ solve \eqref{25.06.19.16.54} on $[t_{n-1},t_{n+1}]$ with the same initial value at time $t_{n-1}$. Therefore, by the localization property in Theorem \ref{thm:local_wp}, they coincide on their common interval of existence. An analogous argument implies that, on the event $\{\tau_{n+1}\geq t_{n+1}\}$, one has 
				$$\sigma = \sigma_{n}$$
				and
				$$ U(t)=U_{n}(t), \qquad t\in [t_{n},t_{n+1}].$$
				We next verify that $\tau_{n+1}$ satisfies the inductive bound
				\eqref{eq:inductiveProb}. 
				By \eqref{26.04.16.16.01},
				\begin{align*}
					\bP(\tau_{n+1}\geq t_{n+1})
					&\geq
					1-\bP(\tau_n<t_n)-\bP(\tau_{n-1}<t_{n-1})-\bP(\lambda_{n-1}<t_{n+1})
					-\bP(\lambda_n<t_{n+1}) \\
					&>
					1-2^n\varepsilon_0-2^{n-1}\varepsilon_0-\varepsilon_0-\varepsilon_0 \\
					&\geq
					1-2^{n+1}\varepsilon_0
					=
					1-2^{n+1-N_0}\varepsilon,
				\end{align*}
				where in the last inequality we used $n\geq2$.
				
				Moreover, using \eqref{eq:boundTraceLambda}, we obtain
				\begin{align*}
					\bE\Bigl[
					\mathbf 1_{\{\tau_{n+1}\geq t_{n+1}\}}
					\|U(t_{n+1})\|_{X_{1-\frac{1+\kappa}{p},p}}^p
					\Bigr]
					&=
					\bE\Bigl[
					\mathbf 1_{\{\tau_{n+1}\geq t_{n+1}\}}
					\|U_{n-1}(t_{n+1})\|_{X_{1-\frac{1+\kappa}{p},p}}^p
					\Bigr]\\
					&\leq
					\bE\Bigl[
					\mathbf 1_{\{\lambda_{n-1}\geq t_{n+1}\}}
					\sup_{t\in[t_{n-1},t_{n+1}]}
					\|U_{n-1}(t)\|_{X_{1-\frac{1+\kappa}{p},p}}^p
					\Bigr]\\
					&\leq
					C_1^{\,n-N_0}C_*\varepsilon_0
					\leq
					C_1^{\,n+1-N_0}C_*\varepsilon_0,
				\end{align*}
				since $C_1\geq 1$. This shows the bound \eqref{eq:inductiveTrace}.
				
				Finally, we prove the regularity estimate \eqref{eq:inductiveReg}. To estimate the $H^{\theta,p}$-norm on $(0,t_{n+1})$, we localize the time interval by means of a smooth partition of unity. Choose $\chi_1,\chi_2\in C^\infty([0,t_{n+1}])$ such that
				$$
				\chi_1+\chi_2=1
				\qquad\text{on }[0,t_{n+1}],
				$$
				with
				$$
				\operatorname{supp}\chi_1\subset [0,t_n),
				\quad \text{ and } \quad 
				\operatorname{supp}\chi_2\subset \Bigl(\frac{t_{n-1}+t_n}{2},\,t_{n+1}\Bigr].
				$$
				Then
				\begin{equation}\label{26.04.16.16.20}
					U=\chi_1U+\chi_2U
					\qquad\text{on }(0,t_{n+1}).
				\end{equation}
				By \eqref{26.04.16.16.20} and the triangle inequality,
				$$
				\|U\|_{H^{\theta,p}((0,t_{n+1}),w_\kappa;X_{1-\theta})}
				\leq
				\|\chi_1U\|_{H^{\theta,p}((0,t_{n+1}),w_\kappa;X_{1-\theta})}
				+
				\|\chi_2U\|_{H^{\theta,p}((0,t_{n+1}),w_\kappa;X_{1-\theta})}.
				$$
				Since $\chi_1$ is supported in $[0,t_n)$, multiplication by $\chi_1$ only involves the restriction of $U$ to $(0,t_n)$, by the boundedness of multiplication by smooth cutoff functions on $H^{\theta,p}$, it follows that
				$$
				\|\chi_1U\|_{H^{\theta,p}((0,t_{n+1}),w_\kappa;X_{1-\theta})}
				\leq
				C_{\chi_1}\,
				\|U\|_{H^{\theta,p}((0,t_n),w_\kappa;X_{1-\theta})}.
				$$
				Similarly, since $\chi_2$ is supported in $\bigl((t_{n-1}+t_n)/2,t_{n+1}\bigr]$, we obtain
				$$
				\|\chi_2U\|_{H^{\theta,p}((0,t_{n+1}),w_\kappa;X_{1-\theta})}
				\leq
				C_{\chi_2}\,
				\|U\|_{H^{\theta,p}(((t_{n-1}+t_n)/2,t_{n+1}),w_\kappa;X_{1-\theta})}.
				$$
				Combining the above estimates, we arrive at
				$$
				\|U\|_{H^{\theta,p}((0,t_{n+1}),w_\kappa;X_{1-\theta})}^p
				\lesssim
				\|U\|_{H^{\theta,p}((0,t_n),w_\kappa;X_{1-\theta})}^p
				+
				\|U\|_{H^{\theta,p}(((t_{n-1}+t_n)/2,t_{n+1}),w_\kappa;X_{1-\theta})}^p.
				$$
				Since \eqref{26.04.16.16.01} implies
				$$
				\{\tau_{n+1}\geq t_{n+1}\}\subseteq \{\tau_n\geq t_n\},
				$$
				we have
				$$
				\mathbf 1_{\{\tau_{n+1}\geq t_{n+1}\}}
				\|U\|_{H^{\theta,p}((0,t_n),w_\kappa;X_{1-\theta})}^p
				\leq
				\mathbf 1_{\{\tau_n\geq t_n\}}
				\|U\|_{H^{\theta,p}((0,t_n),w_\kappa;X_{1-\theta})}^p.
				$$
				For the second term, we use the inclusion
				$$
				\{\tau_{n+1}\geq t_{n+1}\}\subset \{\lambda_{n-1}\geq t_{n+1}\}.
				$$
				Moreover, on the event $\{\tau_{n+1}\geq t_{n+1}\}$, the localization property in Theorem \ref{thm:local_wp} yields
				$$
				U=U_{n-1}
				\qquad\text{on }[t_{n-1},t_{n+1}].
				$$
				Therefore,
				\begin{align*}
					&\mathbf 1_{\{\tau_{n+1}\geq t_{n+1}\}}
					\|U\|_{H^{\theta,p}(((t_{n-1}+t_n)/2,t_{n+1}),w_\kappa;X_{1-\theta})}^p\\&
					\leq{}
					\mathbf 1_{\{\lambda_{n-1}\geq t_{n+1}\}}
					\|U_{n-1}\|_{H^{\theta,p}(((t_{n-1}+t_n)/2,t_{n+1}),w_\kappa;X_{1-\theta})}^p\\&
					\lesssim
					\mathbf 1_{\{\lambda_{n-1}\geq t_{n+1}\}}
					\|U_{n-1}\|_{H^{\theta,p}((t_{n-1},t_{n+1}),w_\kappa^{t_{n-1}};X_{1-\theta})}^p.
				\end{align*}
				The last inequality follows from the fact that the interval $((t_{n-1}+t_n)/2,t_{n+1})$ is contained in $(t_{n-1},t_{n+1})$, and hence its $H^{\theta,p}$-norm is controlled by the corresponding norm on the larger interval. Taking expectations in the previous estimate and applying the time-shifted version of Lemma \ref{lemma:smallData}, we obtain
				$$
				\bE\Bigl[
				\mathbf 1_{\{\lambda_{n-1}\geq t_{n+1}\}}
				\|U_{n-1}\|_{H^{\theta,p}((t_{n-1},t_{n+1}),w_\kappa^{t_{n-1}};X_{1-\theta})}^p
				\Bigr]
				\leq
				K_\theta \,\bE \|V_{n-1}\|_{X_{1-\frac{1+\kappa}{p},p}}^p.
				$$
				On the other hand, the truncated datum $V_{n-1}$ can be controlled directly by the inductive regularity bound. 
				Indeed, applying the trace embedding on $[0,t_{n-1}]$ together with \eqref{eq:inductiveReg} for $\theta=\theta_0$ and $\theta=0$, we obtain
				$$
				\bE\Bigl[
				\mathbf 1_{\{\tau_{n-1}\geq t_{n-1}\}}
				\|U(t_{n-1})\|_{X_{1-\frac{1+\kappa}{p},p}}^p
				\Bigr]
				\lesssim
				\bE\|U_0\|_{X_{1-\frac{1+\kappa}{p},p}}^p.
				$$
				Since
				$$
				V_{n-1}
				=
				\mathbf 1_{\{\tau_{n-1}\geq t_{n-1}\}}U(t_{n-1}),
				$$
				it follows that
				$$
				\bE \|V_{n-1}\|_{X_{1-\frac{1+\kappa}{p},p}}^p
				\lesssim
				\bE\|U_0\|_{X_{1-\frac{1+\kappa}{p},p}}^p.
				$$
				Consequently,
				$$
				\bE\Bigl[
				\mathbf 1_{\{\lambda_{n-1}\geq t_{n+1}\}}
				\|U_{n-1}\|_{H^{\theta,p}((t_{n-1},t_{n+1}),w_\kappa^{t_{n-1}};X_{1-\theta})}^p
				\Bigr]
				\lesssim_\theta
				\bE\|U_0\|_{X_{1-\frac{1+\kappa}{p},p}}^p.
				$$
				This closes the induction and completes the proof.
				
			}

		\end{proof}

		\vspace{.2cm}
		\noindent
		\textbf{Acknowledgements}
		
		\noindent
		The authors are grateful to Mark Veraar and Antonio Agresti for hosting an Oberwolfach workshop on Stochastic Partial Differential Equations in Critical Spaces between 8th--13th June 2025, where the idea of this project was initiated.

		J.-H.\ Choi was supported by a KIAS Individual Grant (MG102701) at Korea Institute for Advanced Study.
		A. Kumar was supported by the Austrian Science Fund (FWF) \href{https://www.fwf.ac.at/forschungsradar/10.55776/ESP4373225}{10.55776/ESP4373225}.
		A.\ Pitrone was supported by the Clarendon Fund in partnership with a Mary Somerville Graduate Scholarship and a Scatcherd European Scholarship.
		S.\ Popat was supported by the EPSRC Centre for Doctoral Training in Mathematics of Random Systems: Analysis, Modelling and Simulation (EP/S023925/1), and now acknowledges funding from the SDAIM project ANR-22-CE40-0015 funded by the French National Research Agency (ANR).

		\bibliographystyle{alpha}
		\bibliography{references}

@phdthesis{FH22,
  title={Stochastic {PDE}s with cross-diffusion effects},
  author={Huber, Florian},
  year={2022},
  school={Technische Universit{\"a}t Wien}
}

@article{fournier2023particle,
  title={Particle approximation of the doubly parabolic {K}eller-{S}egel equation in the plane},
  author={Fournier, Nicolas and Toma{\v{s}}evi{\'c}, Milica},
  journal={Journal of Functional Analysis},
  volume={285},
  number={7},
  pages={110064},
  year={2023},
  publisher={Elsevier}
}

@article {DH03,
	AUTHOR = {Horstmann, Dirk},
	TITLE = {From 1970 until present: the {K}eller-{S}egel model in
	chemotaxis and its consequences. {I}},
	JOURNAL = {Jahresber. Deutsch. Math.-Verein.},
	FJOURNAL = {Jahresbericht der Deutschen Mathematiker-Vereinigung},
	VOLUME = {105},
	YEAR = {2003},
	NUMBER = {3},
	PAGES = {103--165},
	ISSN = {0012-0456},
	MRCLASS = {35K57 (35B30 35J20 35J60 92C17)},
	MRNUMBER = {2013508},
}

@article {DJZ19,
    AUTHOR = {Dhariwal, Gaurav and J\"ungel, Ansgar and Zamponi, Nicola},
     TITLE = {Global martingale solutions for a stochastic population
              cross-diffusion system},
   JOURNAL = {Stochastic Process. Appl.},
  FJOURNAL = {Stochastic Processes and their Applications},
    VOLUME = {129},
      YEAR = {2019},
    NUMBER = {10},
     PAGES = {3792--3820},
      ISSN = {0304-4149,1879-209X},
   MRCLASS = {60H15 (35K51 35R60 60J10 92D25)},
  MRNUMBER = {3997662},
       DOI = {10.1016/j.spa.2018.11.001},
       URL = {https://doi.org/10.1016/j.spa.2018.11.001},
}

@article {DHJN21,
    AUTHOR = {Dhariwal, Gaurav and Huber, Florian and J\"ungel, Ansgar and
              Kuehn, Christian and Neam\c tu, Alexandra},
     TITLE = {Global martingale solutions for quasilinear {SPDE}s via the
              boundedness-by-entropy method},
   JOURNAL = {Ann. Inst. Henri Poincar\'e{} Probab. Stat.},
  FJOURNAL = {Annales de l'Institut Henri Poincar\'e{} Probabilit\'es et
              Statistiques},
    VOLUME = {57},
      YEAR = {2021},
    NUMBER = {1},
     PAGES = {577--602},
      ISSN = {0246-0203,1778-7017},
   MRCLASS = {60H15 (35R60)},
  MRNUMBER = {4255185},
MRREVIEWER = {Le\ Chen},
       DOI = {10.1214/20-aihp1088},
       URL = {https://doi.org/10.1214/20-aihp1088},
}

@article {BHJ24,
    AUTHOR = {Braukhoff, Marcel and Huber, Florian and J\"ungel, Ansgar},
     TITLE = {Global martingale solutions for stochastic
              {S}higesada-{K}awasaki-{T}eramoto population models},
   JOURNAL = {Stoch. Partial Differ. Equ. Anal. Comput.},
  FJOURNAL = {Stochastics and Partial Differential Equations. Analysis and
              Computations},
    VOLUME = {12},
      YEAR = {2024},
    NUMBER = {1},
     PAGES = {525--575},
      ISSN = {2194-0401,2194-041X},
   MRCLASS = {60H15 (35Q92 35R60 92D25)},
  MRNUMBER = {4709549},
MRREVIEWER = {Jiang\ Lun\ Wu},
       DOI = {10.1007/s40072-023-00289-7},
       URL = {https://doi.org/10.1007/s40072-023-00289-7},
}

@article {AV25,
    AUTHOR = {Agresti, Antonio and Veraar, Mark},
     TITLE = {Nonlinear {SPDE}s and maximal regularity: an extended survey},
   JOURNAL = {NoDEA Nonlinear Differential Equations Appl.},
  FJOURNAL = {NoDEA. Nonlinear Differential Equations and Applications},
    VOLUME = {32},
      YEAR = {2025},
    NUMBER = {6},
     PAGES = {Paper No. 123, 150},
      ISSN = {1021-9722,1420-9004},
   MRCLASS = {60H15 (35K57 35K90 47D06 76M35)},
  MRNUMBER = {4952170},
       DOI = {10.1007/s00030-025-01090-2},
       URL = {https://doi.org/10.1007/s00030-025-01090-2},
}

@article {DPL98,
	AUTHOR = {Da Prato, Giuseppe and Lunardi, Alessandra},
	TITLE = {Maximal regularity for stochastic convolutions in {$L^p$}
	spaces},
	JOURNAL = {Atti Accad. Naz. Lincei Cl. Sci. Fis. Mat. Natur. Rend. Lincei
	(9) Mat. Appl.},
	FJOURNAL = {Atti della Accademia Nazionale dei Lincei. Classe di Scienze
	Fisiche, Matematiche e Naturali. Rendiconti Lincei. Serie IX.
	Matematica e Applicazioni},
	VOLUME = {9},
	YEAR = {1998},
	NUMBER = {1},
	PAGES = {25--29},
	ISSN = {1120-6330,1720-0768},
	MRCLASS = {60H05 (60H15 60H25)},
	MRNUMBER = {1669252},
	MRREVIEWER = {H.-J.\ Engelbert},
}

@article {NVW12a,
	AUTHOR = {van Neerven, Jan and Veraar, Mark and Weis, Lutz},
	TITLE = {Stochastic maximal {$L^p$}-regularity},
	JOURNAL = {Ann. Probab.},
	FJOURNAL = {The Annals of Probability},
	VOLUME = {40},
	YEAR = {2012},
	NUMBER = {2},
	PAGES = {788--812},
	ISSN = {0091-1798,2168-894X},
	MRCLASS = {60H15 (35B65 42B25 47A60 47D06)},
	MRNUMBER = {2952092},
	MRREVIEWER = {Feng-Yu\ Wang},
	DOI = {10.1214/10-AOP626},
	URL = {https://doi.org/10.1214/10-AOP626},
}

@article {NVW12b,
	AUTHOR = {van Neerven, Jan and Veraar, Mark and Weis, Lutz},
	TITLE = {Maximal {$L^p$}-regularity for stochastic evolution equations},
	JOURNAL = {SIAM J. Math. Anal.},
	FJOURNAL = {SIAM Journal on Mathematical Analysis},
	VOLUME = {44},
	YEAR = {2012},
	NUMBER = {3},
	PAGES = {1372--1414},
	ISSN = {0036-1410,1095-7154},
	MRCLASS = {60H15 (35R60 46B09 47D06)},
	MRNUMBER = {2982717},
	MRREVIEWER = {Elisa\ Al\`os},
	DOI = {10.1137/110832525},
	URL = {https://doi.org/10.1137/110832525},
}

@book {K99,
	TITLE = {Stochastic partial differential equations: six perspectives},
	SERIES = {Mathematical Surveys and Monographs},
	VOLUME = {64},
		AUTHOR = {Krylov, Nicolai V.},
	EDITOR = {Carmona, Rene A. and Rozovskii, Boris},
	PUBLISHER = {American Mathematical Society, Providence, RI, 185--242},
	YEAR = {1999},
	PAGES = {185--242},
	ISBN = {0-8218-0806-0},
	MRCLASS = {60-06 (35-06 35R60 60H15)},
	MRNUMBER = {1661761},
	DOI = {10.1090/surv/064},
	URL = {https://doi.org/10.1090/surv/064},
}

@article {K96,
	AUTHOR = {Krylov, Nicolai V.},
	TITLE = {On {$L_p$}-theory of stochastic partial differential equations
	in the whole space},
	JOURNAL = {SIAM J. Math. Anal.},
	FJOURNAL = {SIAM Journal on Mathematical Analysis},
	VOLUME = {27},
	YEAR = {1996},
	NUMBER = {2},
	PAGES = {313--340},
	ISSN = {0036-1410},
	MRCLASS = {60H15 (35R60)},
	MRNUMBER = {1377477},
	MRREVIEWER = {Krystyna\ Twardowska},
	DOI = {10.1137/S0036141094263317},
	URL = {https://doi.org/10.1137/S0036141094263317},
}

@article {SJW19,
	AUTHOR = {Shang, Yadong and Tian, Jianjun Paul and Wang, Bixiang},
	TITLE = {Asymptotic behavior of the stochastic {K}eller-{S}egel
	equations},
	JOURNAL = {Discrete Contin. Dyn. Syst. Ser. B},
	FJOURNAL = {Discrete and Continuous Dynamical Systems. Series B. A Journal
	Bridging Mathematics and Sciences},
	VOLUME = {24},
	YEAR = {2019},
	NUMBER = {3},
	PAGES = {1367--1391},
	ISSN = {1531-3492,1553-524X},
	MRCLASS = {35R60 (35B40 35K40 37L30 60H15)},
	MRNUMBER = {3918161},
	DOI = {10.3934/dcdsb.2019020},
	URL = {https://doi.org/10.3934/dcdsb.2019020},
}

@article {MST22,
	AUTHOR = {Misiats, Oleksandr and Stanzhytskyi, Oleksandr and Topaloglu,
	Ihsan},
	TITLE = {On global existence and blowup of solutions of stochastic
	{K}eller-{S}egel type equation},
	JOURNAL = {NoDEA Nonlinear Differential Equations Appl.},
	FJOURNAL = {NoDEA. Nonlinear Differential Equations and Applications},
	VOLUME = {29},
	YEAR = {2022},
	NUMBER = {1},
	PAGES = {Paper No. 3, 29},
	ISSN = {1021-9722,1420-9004},
	MRCLASS = {35B44 (35K55 60H30 65M75)},
	MRNUMBER = {4344836},
	DOI = {10.1007/s00030-021-00735-2},
	URL = {https://doi.org/10.1007/s00030-021-00735-2},
}

@article {ZL25,
	AUTHOR = {Zhang, Lei and Liu, Bin},
	TITLE = {On the {K}eller-{S}egel models interacting with a
	stochastically forced incompressible viscous flow in
	{$\Bbb{R}^2$}},
	JOURNAL = {J. Differential Equations},
	FJOURNAL = {Journal of Differential Equations},
	VOLUME = {414},
	YEAR = {2025},
	PAGES = {487--554},
	ISSN = {0022-0396,1090-2732},
	MRCLASS = {35Q30 (35Q35 35R60 60H15 76D05 76M35 92C17)},
	MRNUMBER = {4797471},
	MRREVIEWER = {Pawe\l\ Szafraniec},
	DOI = {10.1016/j.jde.2024.09.013},
	URL = {https://doi.org/10.1016/j.jde.2024.09.013},
}

@article {CZZ25,
	AUTHOR = {Chen, Yunfeng and Zhai, Jianliang and Zhang, Tusheng},
	TITLE = {Well-posedness of stochastic chemotaxis system},
	JOURNAL = {J. Differential Equations},
	FJOURNAL = {Journal of Differential Equations},
	VOLUME = {442},
	YEAR = {2025},
	PAGES = {Paper No. 113531, 40},
	ISSN = {0022-0396,1090-2732},
	MRCLASS = {60H15 (35K55 60H30 60H50)},
	MRNUMBER = {4919124},
	MRREVIEWER = {Igor\ V.\ Samoilenko},
	DOI = {10.1016/j.jde.2025.113531},
	URL = {https://doi.org/10.1016/j.jde.2025.113531},
}

@article {HMT22a,
	AUTHOR = {Hausenblas, Erika and Mukherjee, Debopriya and Tran, Thanh},
	TITLE = {The one-dimensional stochastic {K}eller-{S}egel model with
	time-homogeneous spatial {W}iener processes},
	JOURNAL = {J. Differential Equations},
	FJOURNAL = {Journal of Differential Equations},
	VOLUME = {310},
	YEAR = {2022},
	PAGES = {506--554},
	ISSN = {0022-0396,1090-2732},
	MRCLASS = {60H15 (35A01 35B65 35K51 35K87 35Q92 92C17)},
	MRNUMBER = {4355927},
	MRREVIEWER = {Stefano\ Bonaccorsi},
	DOI = {10.1016/j.jde.2021.10.056},
	URL = {https://doi.org/10.1016/j.jde.2021.10.056},
}

@article{HMT22b,
	title={Uniqueness of the stochastic keller-segel model in one dimension},
	author={Hausenblas, Erika and Mukherjee, Debopriya and Tran, Thanh},
	journal={arXiv preprint arXiv:2209.13188},
	year={2022}
}

@article {AS00,
	AUTHOR = {Stevens, Angela},
	TITLE = {The derivation of chemotaxis equations as limit dynamics of
	moderately interacting stochastic many-particle systems},
	JOURNAL = {SIAM J. Appl. Math.},
	FJOURNAL = {SIAM Journal on Applied Mathematics},
	VOLUME = {61},
	YEAR = {2000},
	NUMBER = {1},
	PAGES = {183--212},
	ISSN = {0036-1399,1095-712X},
	MRCLASS = {92C17 (60K99 92D50)},
	MRNUMBER = {1776393},
	MRREVIEWER = {John\ G.\ Milton},
	DOI = {10.1137/S0036139998342065},
	URL = {https://doi.org/10.1137/S0036139998342065},
}

@article{BP04,
	title={P{D}{E} models for chemotactic movements: parabolic, hyperbolic and kinetic},
	author={Perthame, Benoit},
    journal={App. Math.},
	fjournal={Applications of Mathematics},
	volume={49},
	number={6},
	pages={539--564},
	year={2004},
	publisher={Springer}
}

@article {PB18,
	AUTHOR = {Biler, Piotr},
	TITLE = {Mathematical challenges in the theory of chemotaxis},
	JOURNAL = {Ann. Math. Sil.},
	FJOURNAL = {Annales Mathematicae Silesianae},
	VOLUME = {32},
	YEAR = {2018},
	NUMBER = {1},
	PAGES = {43--63},
	ISSN = {0860-2107,2391-4238},
	MRCLASS = {92C17 (35M31 35Q92)},
	MRNUMBER = {3846756},
	DOI = {10.2478/amsil-2018-0004},
	URL = {https://doi.org/10.2478/amsil-2018-0004},
}

@article {HP09,
	AUTHOR = {Hillen, Thomas and Painter, Kevin John},
	TITLE = {A user's guide to {PDE} models for chemotaxis},
	JOURNAL = {J. Math. Biol.},
	FJOURNAL = {Journal of Mathematical Biology},
	VOLUME = {58},
	YEAR = {2009},
	NUMBER = {1-2},
	PAGES = {183--217},
	ISSN = {0303-6812,1432-1416},
	MRCLASS = {92C17 (35K57)},
	MRNUMBER = {2448428},
	DOI = {10.1007/s00285-008-0201-3},
	URL = {https://doi.org/10.1007/s00285-008-0201-3},
}

@article {BB15,
	AUTHOR = {Bellomo, Nicola and Bellouquid, Abdelghani and Tao, Youshan and Winkler, Michael},
	TITLE = {Toward a mathematical theory of {K}eller-{S}egel models of
	pattern formation in biological tissues},
	JOURNAL = {Math. Models Methods Appl. Sci.},
	FJOURNAL = {Mathematical Models and Methods in Applied Sciences},
	VOLUME = {25},
	YEAR = {2015},
	NUMBER = {9},
	PAGES = {1663--1763},
	ISSN = {0218-2025,1793-6314},
	MRCLASS = {35K51 (35B30 35B40 35B44 35K57 35Q35 35Q92 92C17)},
	MRNUMBER = {3351175},
	DOI = {10.1142/S021820251550044X},
	URL = {https://doi.org/10.1142/S021820251550044X},
}

@article{DH04,
    AUTHOR = {Horstmann, Dirk},
     TITLE = {From 1970 until present: the {K}eller-{S}egel model in
              chemotaxis and its consequences. {II}},
   JOURNAL = {Jahresber. Deutsch. Math.-Verein.},
  FJOURNAL = {Jahresbericht der Deutschen Mathematiker-Vereinigung},
    VOLUME = {106},
      YEAR = {2004},
    NUMBER = {2},
     PAGES = {51--69},
      ISSN = {0012-0456},
   MRCLASS = {92C17},
  MRNUMBER = {2073515},
}

@article {CP53,
	AUTHOR = {Patlak, Clifford S.},
	TITLE = {Random walk with persistence and external bias},
	JOURNAL = {Bull. Math. Biophys.},
	FJOURNAL = {The Bulletin of Mathematical Biophysics},
	VOLUME = {15},
	YEAR = {1953},
	PAGES = {311--338},
	ISSN = {0007-4985,1522-9602},
	MRCLASS = {60.0X},
	MRNUMBER = {81586},
	MRREVIEWER = {H.\ P.\ Edmundson},
	DOI = {10.1007/bf02476407},
	URL = {https://doi.org/10.1007/bf02476407},
}

@article {KS70,
	AUTHOR = {Keller, Evelyn Fox and Segel, Lee A.},
	TITLE = {Initiation of slime mold aggregation viewed as an instability},
	JOURNAL = {J. Theoret. Biol.},
	FJOURNAL = {Journal of Theoretical Biology},
	VOLUME = {26},
	YEAR = {1970},
	NUMBER = {3},
	PAGES = {399--415},
	ISSN = {0022-5193,1095-8541},
	MRCLASS = {92C17 (35Q92 82C24 92C45)},
	MRNUMBER = {3925816},
	DOI = {10.1016/0022-5193(70)90092-5},
	URL = {https://doi.org/10.1016/0022-5193(70)90092-5},
}

@article {HJ2021,
	AUTHOR = {Huang, Hui and Qiu, Jinniao},
	TITLE = {The microscopic derivation and well-posedness of the
	stochastic {K}eller-{S}egel equation},
	JOURNAL = {J. Nonlinear Sci.},
	FJOURNAL = {Journal of Nonlinear Science},
	VOLUME = {31},
	YEAR = {2021},
	NUMBER = {1},
	PAGES = {Paper No. 6, 31},
	ISSN = {0938-8974,1432-1467},
	MRCLASS = {65M75 (35K55 60H15 60H30)},
	MRNUMBER = {4192425},
	MRREVIEWER = {Arbaz\ Khan},
	DOI = {10.1007/s00332-020-09661-6},
	URL = {https://doi.org/10.1007/s00332-020-09661-6},
}

@incollection{da1985maximal,
  title={Maximal regularity for stochastic convolutions and applications to stochastic evolution equations in Hilbert spaces},
  author={Da Prato, G},
  booktitle={Stochastic Space—Time Models and Limit Theorems},
  pages={41--52},
  year={1985},
  publisher={Springer}
}

@article{bosch2024multidimensional,
author = {van den Bosch, M. and Hupkes, Hermen Jan},
title = {Multidimensional Stability of Planar Traveling Waves for Stochastically Perturbed Reaction–Diffusion Systems},
journal = {Stud. Appl. Math.},
fjournal = {Studies in Applied Mathematics},
volume = {155},
number = {3},
pages = {e70114},
doi = {https://doi.org/10.1111/sapm.70114},
url = {https://onlinelibrary.wiley.com/doi/abs/10.1111/sapm.70114},
eprint = {https://onlinelibrary.wiley.com/doi/pdf/10.1111/sapm.70114},
year = {2025}
}

@article {Agresti2025StochasticPrimative,
    AUTHOR = {Agresti, Antonio and Hieber, Matthias and Hussein, Amru and
              Saal, Martin},
     TITLE = {The stochastic primitive equations with nonisothermal
              turbulent pressure},
   JOURNAL = {Ann. Appl. Probab.},
  FJOURNAL = {The Annals of Applied Probability},
    VOLUME = {35},
      YEAR = {2025},
    NUMBER = {1},
     PAGES = {635--700},
      ISSN = {1050-5164,2168-8737},
   MRCLASS = {35Q86 (35Q35 35R60 60H15 76M35 76U60)},
  MRNUMBER = {4871718},
MRREVIEWER = {Adrian\ Muntean},
       DOI = {10.1214/24-aap2124},
       URL = {https://doi.org/10.1214/24-aap2124},
}

@article {Agresti2024StochasticNS,
    AUTHOR = {Agresti, Antonio and Veraar, Mark},
     TITLE = {Stochastic {N}avier-{S}tokes equations for turbulent flows in
              critical spaces},
   JOURNAL = {Comm. Math. Phys.},
  FJOURNAL = {Communications in Mathematical Physics},
    VOLUME = {405},
      YEAR = {2024},
    NUMBER = {2},
     PAGES = {Paper No. 43, 57},
      ISSN = {0010-3616,1432-0916},
   MRCLASS = {35Q30 (35D35 35R60 60H15 76F02 76M35)},
  MRNUMBER = {4703457},
MRREVIEWER = {Zijin\ Li},
       DOI = {10.1007/s00220-023-04867-7},
       URL = {https://doi.org/10.1007/s00220-023-04867-7},
}

@article {Mayorcas2023Blowup,
    AUTHOR = {Mayorcas, Avi and Toma\v{s}evi\'{c}, Milica},
     TITLE = {Blow-up for a stochastic model of chemotaxis driven by
              conservative noise on {$\Bbb {R}^2$}},
   JOURNAL = {J. Evol. Equ.},
  FJOURNAL = {Journal of Evolution Equations},
    VOLUME = {23},
      YEAR = {2023},
    NUMBER = {3},
     PAGES = {Paper No. 57, 28},
      ISSN = {1424-3199,1424-3202},
   MRCLASS = {35B44 (35Q92 35R60 60H15 92C17)},
  MRNUMBER = {4629495},
       DOI = {10.1007/s00028-023-00900-3},
       URL = {https://doi.org/10.1007/s00028-023-00900-3},
}

@article {Lorz2012Coupled,
    AUTHOR = {Lorz, Alexander},
     TITLE = {A coupled {K}eller-{S}egel-{S}tokes model: global existence
              for small initial data and blow-up delay},
   JOURNAL = {Commun. Math. Sci.},
  FJOURNAL = {Communications in Mathematical Sciences},
    VOLUME = {10},
      YEAR = {2012},
    NUMBER = {2},
     PAGES = {555--574},
      ISSN = {1539-6746,1945-0796},
   MRCLASS = {35Q92 (35A01 35B44 35M31)},
  MRNUMBER = {2901320},
       DOI = {10.4310/CMS.2012.v10.n2.a7},
       URL = {https://doi.org/10.4310/CMS.2012.v10.n2.a7},
}

@article {Winkler2019Three,
    AUTHOR = {Winkler, Michael},
     TITLE = {A three-dimensional {K}eller-{S}egel-{N}avier-{S}tokes system
              with logistic source: global weak solutions and asymptotic
              stabilization},
   JOURNAL = {J. Funct. Anal.},
  FJOURNAL = {Journal of Functional Analysis},
    VOLUME = {276},
      YEAR = {2019},
    NUMBER = {5},
     PAGES = {1339--1401},
      ISSN = {0022-1236,1096-0783},
   MRCLASS = {35Q35 (35B40 35D30 35K55 92C17)},
  MRNUMBER = {3912779},
MRREVIEWER = {Christian\ Stinner},
       DOI = {10.1016/j.jfa.2018.12.009},
       URL = {https://doi.org/10.1016/j.jfa.2018.12.009},
}

@article {Martini2025Additive,
    AUTHOR = {Martini, Adrian and Mayorcas, Avi},
     TITLE = {An additive-noise approximation to
              {K}eller-{S}egel-{D}ean-{K}awasaki dynamics: local
              well-posedness of paracontrolled solutions},
   JOURNAL = {Stoch. Partial Differ. Equ. Anal. Comput.},
  FJOURNAL = {Stochastics and Partial Differential Equations. Analysis and
              Computations},
    VOLUME = {13},
      YEAR = {2025},
    NUMBER = {2},
     PAGES = {956--1033},
      ISSN = {2194-0401,2194-041X},
   MRCLASS = {60H17 (60L40 92C17)},
  MRNUMBER = {4908981},
       DOI = {10.1007/s40072-024-00343-y},
       URL = {https://doi.org/10.1007/s40072-024-00343-y},
}

@article{AV24,
  title={Stochastic maximal ${L}^p ({L}^q)$-regularity for second order systems with periodic boundary conditions},
  author={Agresti, Antonio and Veraar, Mark},
  JOURNAL = {Ann. Inst. Henri Poincaré Probab. Stat.},
  FJournal={Annales de l'Institut Henri Poincare (B) Probabilites et statistiques},
  volume={60},
  number={1},
  pages={413--430},
  year={2024},
  organization={Institut Henri Poincar{\'e}}
}

@book{hytonen2018analysis,
  title={Analysis in Banach Spaces: Volume II: Probabilistic Methods and Operator Theory},
  author={Hyt{\"o}nen, Tuomas and Van Neerven, Jan and Veraar, Mark and Weis, Lutz},
  volume={67},
  year={2018},
  publisher={Springer}
}

@article {AV_SQEE_pt2,
    AUTHOR = {Agresti, Antonio and Veraar, Mark},
     TITLE = {Nonlinear parabolic stochastic evolution equations in critical
              spaces part {II}: {B}low-up criteria and instataneous
              regularization},
   JOURNAL = {J. Evol. Equ.},
  FJOURNAL = {Journal of Evolution Equations},
    VOLUME = {22},
      YEAR = {2022},
    NUMBER = {2},
     PAGES = {Paper No. 56, 96},
      ISSN = {1424-3199,1424-3202},
   MRCLASS = {35R60 (35K59 60H15)},
  MRNUMBER = {4437443},
MRREVIEWER = {Le\ Chen},
       DOI = {10.1007/s00028-022-00786-7},
       URL = {https://doi.org/10.1007/s00028-022-00786-7},
}

@article {AV_SQEE_pt1,
    AUTHOR = {Agresti, Antonio and Veraar, Mark},
     TITLE = {Nonlinear parabolic stochastic evolution equations in critical
              spaces part {I}. {S}tochastic maximal regularity and local
              existence},
   JOURNAL = {Nonlinearity},
  FJOURNAL = {Nonlinearity},
    VOLUME = {35},
      YEAR = {2022},
    NUMBER = {8},
     PAGES = {4100--4210},
      ISSN = {0951-7715,1361-6544},
   MRCLASS = {60H15 (35K59 35R60 42B37 47D06)},
  MRNUMBER = {4459102},
       DOI = {10.1088/1361-6544/abd613},
       URL = {https://doi.org/10.1088/1361-6544/abd613},
}

@article{ALV,
author = {Agresti, Antonio and Lindemulder, Nick and Veraar, Mark},
title = {On the trace embedding and its applications to evolution equations},
journal = {Mathematische Nachrichten},
volume = {296},
number = {4},
pages = {1319-1350},
doi = {https://doi.org/10.1002/mana.202100192},
year = {2023}
}

@book {schmeisser,
    AUTHOR = {Schmeisser, Hans-J\"urgen and Triebel, Hans},
     TITLE = {Topics in {F}ourier analysis and function spaces},
    VOLUME = {42},
 PUBLISHER = {Akademische Verlagsgesellschaft Geest \& Portig K.-G.,
              Leipzig},
      YEAR = {1987},
     PAGES = {300},
      ISBN = {3-321-00001-6},
   MRCLASS = {42B30 (42-02 46E35 46M35)},
  MRNUMBER = {900143},
MRREVIEWER = {Mario\ Milman},
}

@article{FGL21,
author = {Franco Flandoli and Lucio Galeati and Dejun Luo},
title = {Delayed blow-up by transport noise},
journal = {Communications in Partial Differential Equations},
volume = {46},
number = {9},
pages = {1757--1788},
year = {2021},
publisher = {Taylor \& Francis},
doi = {10.1080/03605302.2021.1893748}}

@article{hornung2019quasilinear,
  title={Quasilinear parabolic stochastic evolution equations via maximal ${L}^p$-regularity},
  author={Hornung, Luca},
   journal={Potential Anal.},
  fjournal={Potential Analysis},
  volume={50},
  number={2},
  pages={279--326},
  year={2019},
  publisher={Springer}
}

@book{da2014stochastic,
  title={Stochastic equations in infinite dimensions},
  author={Da Prato, Giuseppe and Zabczyk, Jerzy},
  year={2014},
  publisher={Cambridge university press}
}

@article{agresti2025stochastic,
  title={A stochastic flow approach to De Giorgi-Nash-Moser estimates for SPDEs with smooth transport noise},
  author={Agresti, Antonio and Sauerbrey, Max and Veraar, Mark},
  journal={arXiv preprint arXiv:2511.12692},
  year={2025}
}

@article{pruss2018critical,
  title={Critical spaces for quasilinear parabolic evolution equations and applications},
  author={Pr{\"u}ss, Jan and Simonett, Gieri and Wilke, Mathias},
  journal= {J. Differential Equations},
  Fjournal={Journal of Differential Equations},
  volume={264},
  number={3},
  pages={2028--2074},
  year={2018},
  publisher={Elsevier}
}

@article{carillo_jungel,
author = {Carrillo, Jos{\'e} Antonio and Hittmeir, Sabine and J\"{u}ngel, Ansgar},
title = {CROSS DIFFUSION AND NONLINEAR DIFFUSION PREVENTING BLOW UP IN THE {K}ELLER–{S}EGEL MODEL},
JOURNAL = {Math. Models Methods Appl. Sci.},
FJournal = {Mathematical Models and Methods in Applied Sciences},
volume = {22},
number = {12},
pages = {1250041},
year = {2012},
doi = {10.1142/S0218202512500418}}

@article{KOZONO20091,
title = {Global strong solution to the semi-linear {K}eller–{S}egel system of parabolic–parabolic type with small data in scale invariant spaces},
journal = {J. of Differential Equations},
FJournal = {Journal of Differential Equations},
volume = {247},
number = {1},
pages = {1--32},
year = {2009},
issn = {0022-0396},
doi = {https://doi.org/10.1016/j.jde.2009.03.027},
url = {https://www.sciencedirect.com/science/article/pii/S0022039609001417},
author = {Hideo Kozono and Yoshie Sugiyama}
}

@article {Arumugam_Tyagi,
    AUTHOR = {Arumugam, Gurusamy and Tyagi, Jagmohan},
     TITLE = {Keller-{S}egel chemotaxis models: a review},
   JOURNAL = {Acta Appl. Math.},
  FJOURNAL = {Acta Applicandae Mathematicae},
    VOLUME = {171},
      YEAR = {2021},
     PAGES = {Paper No. 6, 82},
      ISSN = {0167-8019,1572-9036},
   MRCLASS = {92C17 (35A01 35B44 35D30 35K40 65M06 65M08)},
  MRNUMBER = {4188348},
       DOI = {10.1007/s10440-020-00374-2},
       URL = {https://doi.org/10.1007/s10440-020-00374-2},
}

@article {Bedrossian_He,
    AUTHOR = {Bedrossian, Jacob and He, Siming},
     TITLE = {Suppression of blow-up in {P}atlak-{K}eller-{S}egel via shear
              flows},
   JOURNAL = {SIAM J. Math. Anal.},
  FJOURNAL = {SIAM Journal on Mathematical Analysis},
    VOLUME = {49},
      YEAR = {2017},
    NUMBER = {6},
     PAGES = {4722--4766},
      ISSN = {0036-1410,1095-7154},
   MRCLASS = {35M31 (35B40 35B44 35B45)},
  MRNUMBER = {3730537},
MRREVIEWER = {Jingyu\ Li},
       DOI = {10.1137/16M1093380},
       URL = {https://doi.org/10.1137/16M1093380},
}

@article {Nagai,
    AUTHOR = {Nagai, Toshitaka and Senba, Takasi and Yoshida, Kiyoshi},
     TITLE = {Application of the {T}rudinger-{M}oser inequality to a
              parabolic system of chemotaxis},
   JOURNAL = {Funkcial. Ekvac.},
  FJOURNAL = {Funkcialaj Ekvacioj. Serio Internacia},
    VOLUME = {40},
      YEAR = {1997},
    NUMBER = {3},
     PAGES = {411--433},
      ISSN = {0532-8721},
   MRCLASS = {92D50 (35K55 35Q80 92C99)},
  MRNUMBER = {1610709},
MRREVIEWER = {V.\ N.\ Razzhevaikin},
       DOI = {10.24546/0100499987},
       URL = {https://doi.org/10.24546/0100499987},
}

@article {herrero,
    AUTHOR = {Herrero, Miguel \'Angel and Vel\'azquez, Juan J. L\'opez},
     TITLE = {A blow-up mechanism for a chemotaxis model},
   JOURNAL = {Ann. Scuola Norm. Sup. Pisa Cl. Sci. (4)},
  FJOURNAL = {Annali della Scuola Normale Superiore di Pisa. Classe di
              Scienze. Serie IV},
    VOLUME = {24},
      YEAR = {1997},
    NUMBER = {4},
     PAGES = {633--683},
      ISSN = {0391-173X,2036-2145},
   MRCLASS = {92C15 (35K60 35Q80 92D50)},
  MRNUMBER = {1627338},
MRREVIEWER = {Yoram\ Schiffmann},
       URL = {http://www.numdam.org/item?id=ASNSP_1997_4_24_4_633_0},
}
		
	\end{document}